\documentclass[11pt,a4paper,reqno]{amsart}
\usepackage{tabls}
\usepackage{url}
\usepackage[linktocpage = true]{hyperref}
\hypersetup{
 colorlinks = true,
 citecolor = blue,
}
\usepackage{color}
\definecolor{red}{rgb}{1,0,0}
\definecolor{green}{rgb}{0,1,0}
\definecolor{blue}{rgb}{0,0,1}
\definecolor{refkey}{gray}{.625}
\definecolor{labelkey}{gray}{.625}
\usepackage[lite]{amsrefs}
\usepackage[a4paper]{geometry}
\usepackage{amsfonts}
\usepackage{amssymb}
\usepackage{mathrsfs}
\usepackage{amsmath}
\usepackage{amsthm}
\usepackage{amscd}
\usepackage{enumerate}
\usepackage{paralist}
\usepackage{graphicx}
\usepackage{mathtools}
\usepackage[all,cmtip]{xy}
\usepackage{kantlipsum}
\usepackage{tikz}
\usepackage{tikz-cd}
\usetikzlibrary{decorations.pathmorphing}
\allowdisplaybreaks

\newcommand{\abs}[1]{\lvert#1\rvert}

\newcommand{\At}{\operatorname{At}}
\newcommand{\A}{\mathscr{A}}
\newcommand{\B}{\mathcal{B}}

\newcommand{\id}{\operatorname{id}}

\newcommand{\Z}{\mathbb{Z}}

\newcommand{\g}{\mathfrak{g}}

\makeatletter
 \def\title@font{\normalsize\bfseries}
 \let\ltx@maketitle\@maketitle
 \def\@maketitle{\bgroup%
 \let\ltx@title\@title%
 \def\@\title{\resizebox{\textwidth}{!}{%
  \mbox{\title@font\ltx@title}%
 }}%
 \ltx@maketitle%
 \egroup}
\makeatother

\newtheorem{theorem}{Theorem}[section]

\theoremstyle{plain}% default
\newtheorem{Thm}[theorem]{Theorem}

\newtheorem{Cor}[theorem]{Corollary}

\newtheorem*{theorem*}{Theorem}
\newtheorem{Def}[theorem]{Definition}
\newtheorem{def-prop}[theorem]{Definition-Proposition}
\newtheorem{prop}[theorem]{Proposition}
\newtheorem{prop-def}[theorem]{Proposition-Definition}
\newtheorem{Ex}[theorem]{Example}
\newtheorem{Rem}[theorem]{Remark}

\numberwithin{equation}{section}
\theoremstyle{remark}

\begin{document}
\def\C{\mathbb{C}}
\def\CE{\mathrm{CE}}
\def\D{\mathcal{D}}
\def\E{\mathscr{E}}
\def\F{\mathscr{F}}
\def\H{\textbf{H}}
\def\k{\mathbb{K}}
\def\G{\mathcal{G}}
\def\M{\mathcal{M}}
\def\m{\mathfrak{m}}
\def\t{\mathfrak{t}}
\def\O{\mathcal{O}}
\def\P{\mathcal{P}}
\def\U{\mathcal{U}}
\def\V{\mathscr{V}}
\def\L{\mathcal{L}}
\def\X{\mathbb{X}}
\def\Y{\mathbb{Y}}
\def\spec{\text{spec}}
\def\Im{\text{Im}}
\def\coker{\operatorname{coker}}
\def\Ext{\operatorname{Ext}}
\def\End{\operatorname{End}}
\def\pr{\operatorname{pr}}
\def\id{\operatorname{id}}
\def\Der{\operatorname{Der}}
\def\Hom{\operatorname{Hom}}
\def\Jet{\operatorname{Jet}}
\def\Map{\operatorname{Map}}
\def\Mod{\operatorname{Mod}}
\def\sgn{\operatorname{sgn}}
\def\sh{\operatorname{sh}}

\newcommand{\tot}{\mathrm{tot}}

\newcommand{\img}{\mathrm{img}}
\newcommand{\T}{\mathfrak{L}}
\newcommand{\dg}{\mathrm{dg}}

\def\Kap{\operatorname{Kap}}

\newcommand{\LP}{Loday-Pirashvili }

\newcommand{\CATderivationA}{\operatorname{dgDer}_\A}

\newcommand{\CATLeibnizoneA}{\operatorname{Leib}_\infty(\A)}

\newcommand{\CAThtyAmod}{\operatorname{H}(\operatorname{dg}\A)}
\newcommand{\DG}{\mathrm{dg}}
\newcommand{\dCE}{d_{\CE}}
\newcommand{\Omegag}{\Omega_{\g}}

\title{DG Loday-Pirashvili modules over Lie algebras}

\author{Zhuo Chen}{$^1$}\\
%\address{Zhuo Chen, Department of Mathematics, Tsinghua University, Beijing, 100084, China}
%\email{\href{chenzhuo@mail.tsinghua.edu.cn}{chenzhuo@tsinghua.edu.cn}}
%\thanks{Research partially supported by NSFC grant 12071241.}

\author{Yu Qiao*}%{$^{2, *}$}\\
%\address{Yu Qiao \emph{(corresponding author)}, School of Mathematics and Statistics, Shaanxi Normal University, Xi'an, 710119, China}
%\email{\href{yqiao@snnu.edu.cn}{yqiao@snnu.edu.cn}}
%\thanks{Research partially supported by NSFC grant 11971282.}

\author{Maosong Xiang}
%\address{Maosong Xiang, Center for Mathematical Sciences, Huazhong University of Science and Technology, Wuhan, 430074, China}%, Beijing 100871, China} }
%\email{\href{mailto: msxiang@hust.edu.cn}{msxiang@hust.edu.cn}}
%\thanks{Research partially supported by NSFC grant 11901221.}

\author{Tao Zhang}
%\address{Tao Zhang, College of Mathematics and Information Science, Henan Normal University, Xinxiang, 453002, China}
%\email{\href{mailto:zhangtao@htu.edu.cn}{zhangtao@htu.edu.cn}}
%\thanks{Research partially supported by NSFC grant 11501179.}

\thanks{ZC: \emph{Department of Mathematics, Tsinghua University, Beijing, 100084, China}}
\thanks{YQ: {\em School of Mathematics and Statistics, Shaanxi Normal University, Xi'an 710119, Shaanxi, China}}
\thanks{MX: {\em Center for Mathematical Sciences, School of Mathematics and Statistics,  Huazhong University of Science and Technology, Wuhan, 430074, China}}
\thanks{TZ: \emph{College of Mathematics and Information Science, Henan Normal University, Xinxiang, 453002, China}}
\thanks{*: {Corresponding author}}
\thanks{\emph{Emails}: chenzhuo@tsinghua.edu.cn, ~yqiao@snnu.edu.cn,~msxiang@hust.edu.cn,~zhangtao@htu.edu.cn}

%\thanks{The research was partially supported by NSFC grants 12071241 (Chen), 11971282 (Qiao), 11901221 (Xiang), and 11501179 (Zhang).}

\begin{abstract}
A \LP module over a Lie algebra $\g$ is a Lie algebra object $\bigl(G\xrightarrow{X} \g \bigr)$ in the category of linear maps, or equivalently, a $\g$-module $G$ which admits a  $\g$-equivariant linear map $X:G\to \g$. We study dg \LP modules over Lie algebras, which is a generalization of \LP modules in a natural way, and establish several equivalent characterizations of dg \LP modules. To provide a concise characterization, a dg \LP module is a non-negative  and bounded dg $\g$-module $V $ paired with a weak morphism of dg $\g$-modules  $\alpha\colon V\rightsquigarrow \g$.
Such a dg \LP module   resolves  an arbitrarily specified classical \LP module in the sense that it exists and is unique (up to homotopy).
Dg \LP modules can also be characterized through dg derivations.  This perspective allows the calculation of the corresponding twisted Atiyah classes.
By leveraging the Kapranov functor on the dg derivation arising from a dg \LP module $(V,\alpha)$,   a Leibniz$_\infty[1]$  algebra structure can be derived on $\wedge^\bullet \g^\vee\otimes V[1]$. The binary bracket of this structure corresponds to the twisted Atiyah cocycle.
To exemplify these intricate algebraic structures through specific cases, we  utilize this machinery to a particular type of dg \LP modules stemming from Lie algebra pairs.
\end{abstract}

\maketitle

\subjclass{{\em Mathematics Subject Classification} (2020): Primary 16E45; Secondary 16W50, 17A32, 17B55, 18G35 }
\\
\keywords{{\em Keywords}: dg modules over Lie algebras, \LP modules,  twisted Atiyah classes, Kapranov functor, Leibniz$_\infty[1]$-algebras}

\tableofcontents

\section*{Introduction}\paragraph{\bf The motivation}
The inception of this study is rooted in the realm of differential graded (abbreviated as dg) geometry. Recall that a dg manifold is deemed as a manifold endowed with a graded structure, accompanied by a homological vector field. These entities first appeared over two decades ago in the study of topological quantum field theory~\cite{AKSZ}, which is now commonly referred to as the AKSZ (Alexandrov-Kontsevich-Schwarz-Zaboronsky) construction. Building upon this foundational knowledge, Ciocan-Fontanine and Kapranov meticulously laid out a comprehensive framework for defining dg manifolds~\cite{CFK}. Dg manifolds have a rich structure and play an important role in various fields, including but not limited to gauge theory, deformation quantization, Lie theory, higher structures, and most recently, generalized geometry.

In this note we focus on a specific category of objects originating from Lie algebras. A Lie algebra $\g$ is interpreted as a formal pointed dg manifold from the perspective of formal geometry, as noted by~\cite{KS} --- The Lie structure inherent in $\g$ leads to the creation of a dg manifold, conventionally labeled as $\g[1]$. The function algebra associated with $\g[1]$ is $\wedge^\bullet \g^\vee$, and the corresponding homological vector field is the Chevalley-Eilenberg differential. Within this framework, dg vector bundles over $\g[1]$ exhibit a direct correspondence with dg $\g$-modules. For instance, the adjoint representation of $\g$ is linked with the shifted tangent bundle $T[-1] \g[1] \cong \g\times \g[1]$.

In their study~\cite{LP} of the tensor category of linear maps, Loday and Pirashvili demonstrated that any Lie algebra object within this category can be represented by a linear map $X \colon G \to \g$, where $\g$ is a Lie algebra, $G$ is a $\g$-module, and $X$ is a $\g$-equivariant map. This representation is encapsulated in the form of $\bigl(G \xrightarrow{X} \g \bigr)$,  termed as a {\em \LP module}, which is central to our discussion; see Section \ref{Sec:classicalLPmodules} for detailed discussions. These objects are also referred to as \textit{averaging operators} of Lie algebras~\cite{Aguiar} or \textit{augmented Leibniz algebras} in~\cite{ABRW}; further studies on this topic can be found in~\cites{Cao, PG} with respect to associative algebras.

From the aforementioned dg geometry point of view, a \LP module $\bigl(G\xrightarrow{X} \g \bigr)$ can be rephrased as a dg vector bundle $G\times \g[1]$ over $\g[1]$ whose fiber $G$ concentrates in degree $0$. This construction also includes a morphism $G\times \g[1]\xrightarrow{X^{!}} T[-1]\g[1]$, which is a map of dg vector bundles obtained from the map $X \colon G\to \g$. The map $X^!$ is commonly referred to as an anchor map in differential geometry. Therefore, one may treat \LP modules as ``anchored" dg vector bundles with the stipulation that their fibers are concentrated in degree $0$. In order to explore further features of \LP modules from this dg perspective, it is pertinent to consider how we can relax the original constraints on the degrees of vector spaces and broaden the definition of these modules to align more closely with the dg settings and   the context of differential geometry.

  Based on these ideas  and  as a   homotopy replacement,
  we introduce the concept of a dg \LP module as a modification of the traditional \LP module framework. Specifically, a dg \LP module is defined as a dg vector bundle $V\times \mathfrak{g}[1]$ over $\mathfrak{g}[1]$, together with an anchor map represented by a morphism $V\times \mathfrak{g}[1]\to T[-1]\mathfrak{g}[1]$ of dg vector bundles. This formulation can be expressed in terms of Lie algebras and their modules, as highlighted in Definition \ref{Def: het0}. Additionally, we provide alternative characterizations of dg \LP modules. In essence, a dg \LP module (over $\mathfrak{g}$) comprises a dg $\mathfrak{g}$-module $V$ along with a set of linear maps $V^{k} \otimes \wedge^{k} \mathfrak{g} \rightarrow \mathfrak{g}$, which demonstrate compatibility with the relevant $\mathfrak{g}$-module structures, as stated in Proposition \ref{Prop:otherdescriptionsofLP}. Having established a precise definition, there is no necessity to revisit the original concept. Consequently, we avoid further discussion on their dg geometry counterpart in subsequent sections.

\paragraph{\bf Structure and main results of the paper}
The first Section \ref{Sec:gradedLPmod} of the paper serves as a foundation, beginning with the reiteration of fundamental concepts and the examination of the definition of ordinary \LP modules.
A key focus of this section is the introduction of the dg \LP modules.
It is worth mentioning that an essentially surjective functor $\H^0$ is introduced from the category of dg \LP modules to ordinary \LP modules. The central theorem in this section is Theorem \ref{Thm:foundamentalone}, which states that in the case where $V$ serves as a resolution for a $\g$-module $G$, every {\LP module} structure on $G$ can be lifted (in a unique, homotopy-preserving manner) to a dg \LP module structure on $V$. This lifting effectively represents an inverse image of $\H^0$.

In Section \ref{Sec:twistedAtiyahclassesetc}, we approach the structural properties of dg \LP modules from an alternative angle. By   Proposition \ref{prop1}, we establish that a {dg \LP module} structure on $V$ is essentially equivalent to an $\wedge^\bullet \g^\vee \otimes V^\vee[-1]$-valued dg derivation of $\wedge^\bullet \g^\vee$. This connection was initially introduced in a previous collaborative work by the first and third named authors along with Liu \cite{CLX}, where they introduced the concept of twisted Atiyah class associated with dg derivations of general dg algebras (refer to Proposition-Definition \ref{prop:Atiyah via connection}). Subsequently, our focus shifts towards unraveling the corresponding twisted Atiyah classes emerging from dg \LP modules. The outcomes of this exploration are succinctly presented in Proposition \ref{prop:LeibnizgradedLPmodule}. Remarkably, originating from a {dg \LP module} $V$, we are able to derive a (graded) Leibniz algebra $H^\bullet_{\tot} (\g,A\otimes V)$ (the total cohomology linked with the dg $\g$-module $A\otimes V$ for a specific $\g$-algebra $A$). Additionally, a dg $\g$-module $W$ yields a Leibniz module $H^\bullet_{\tot} (\g,A\otimes W)$ over the aforementioned Leibniz algebra. Notably, all structural mappings stem from twisted Atiyah classes.

 Section \ref{Sec:KapranovLeibnizetc} also expands on the material in  our previous work \cite{CLX}, namely a construction of higher algebraic objects, called Kapranov Leibniz$_\infty[1]$-algebras.  One application of this construction, specifically the Kapranov functor, involves the dg derivations originating from dg \LP modules over Lie algebras. The main Theorem \ref{Main Thm} establishes a Leibniz$_\infty[1]$ $\wedge^\bullet \g^\vee$-algebra structure on the $\wedge^\bullet \g^\vee$-module $\wedge^\bullet \g^\vee\otimes V[1]$. This specific object is significant because it facilitates a unique ``\textit{resolution}" of the Leibniz algebra developed by Loday and Pirashvili in \cite{LP} --- For an ordinary \LP module $\bigl(G\xrightarrow{X} \g \bigr)$, where $G$ is resolved as $V=(V^\bullet, d^V)$ with $G=\ker(d_0^V)\subset V^0$,   Theorem \ref{Thm:foundamentalone} derives a dg \LP module $(V, \alpha)$, acting as a dg lift of $(G,X)$. The outcome of Theorem \ref{Main Thm} is a Leibniz$_\infty[1]$ $\wedge^\bullet \g^\vee$-algebra denoted by $\Kap(\delta_\alpha) = (\B,\{\mathcal{R}^{\nabla}_k\}_{k\geqslant 1})$. This algebra is defined on the $\wedge^\bullet \g^\vee$-module $\B= \wedge^\bullet \g^\vee \otimes V[1]$. Notably, the binary bracket of this structure, when applied to $G[1]=\ker(d_0^V)[1]\subset V^0[1]$, precisely aligns with the standard Leibniz algebra structure found in $G$.

 In the last Section \ref{Sec:appLiepair}, we delve into the application of the aforementioned methods and theories to a specific category of dg \LP modules originating from Lie algebra pairs. A Lie algebra pair $(\T, \g)$ is characterized by an inclusion $\g \subset \T$ of Lie algebras. We demonstrate that such a Lie algebra pair naturally results in a canonical dg \LP module structure being imposed upon the two-term cochain complex $V = ( V^{0}\xrightarrow{d^V} V^1)$, with $V^{0}=\T$ and $V^1=\T/\g$. The differential $d^V$ is represented by the projection $\pr_{\T/\g}$. An intriguing discovery emerges as the original Atiyah class of a Lie algebra pair, as defined by Calaque, C\u{a}ld\u{a}raru, and Tu in \cite{CCT}, integrates seamlessly into the framework of the twisted Atiyah class that emerges from the said dg \LP module. This interplay is further elaborated in Remark~\ref{Rem:originalandtwisted}.

\paragraph{\bf Related works}

In the context of maximal gauged supergravities, the concept of \LP modules is closely linked to embedding tensors, as investigated by Nicolai and Samtleben~\cites{NS1, NS2}. Nicolai and Samtleben define an embedding tensor as a linear mapping denoted as $X$, which operates from the space of gauge fields, denoted as $G$, to the Lie algebra of the rigid symmetry group $P$. The group $P$ is representative of the prevalent duality property observed in the realm of string theory and M-theory. It is noteworthy that the Lie algebra $\g$ possesses a representation on the space of gauge fields $G$. Moreover, a significant physical constraint, stemming from the physics background, imposes that the linear map $X \colon G \to \g$ must adhere to a closure constraint, resulting in the mathematical necessity that the image given by $\g_0:= X(G) \subset \g$ forms a Lie subalgebra, with $\theta$ exhibiting $\g_0$-equivariance. Thus, in the terminology of this paper, the map $X \colon G \to \g_0$ constitutes an \LP module. Additionally, Kotov and Strobl~\cite{KS1} demonstrated that the essential data required for an embedding tensor align precisely with those of a Leibniz algebra, from which the associated ${L_\infty}$ algebra can be deduced. Remarkably, Strobl and Wagemann~\cite{SW} have recently explored enhanced Leibniz algebras originating from higher gauge theories, thereby capturing embedding tensors as a distinctive case.

In studies of gauged supergravities~\cites{dWNS,dWS2}, the invariance of the Lagrangian's kinetic terms necessitates a covariant expression for the field strength, thus leading to the concept of tensor hierarchies. The theory of embedding tensor gives rise to tensor hierarchies, which involve the successive introduction of higher gauge transformations. These hierarchies play a crucial role in modeling the dynamics of the system (see ~\cites{dWNS, dWS, dWS2, dWST, dWST2, dWST3}). Furthermore, these tensor hierarchies can be organized into an ${L_\infty}$ algebra, as demonstrated in works by various researchers~\cites{HS, Lavau, Palm, LP2020, LS2021, LS2023}. It is anticipated that future research will benefit from the insights provided by these mathematical physicists, fostering further exploration and enlightenment in this field.

\paragraph{\bf Acknowledgment}
We appreciate the anonymous referee for offering constructive suggestions to enhance the manuscript presentation.

\section{Dg \LP modules}~\label{Sec:gradedLPmod}

Let us fix {convention and notations}.
Throughout this paper, $\k$ denotes a field of characteristic zero and ``graded'' means $\Z$-graded.
In what follows, the notion ``vector space'' means ``$\k$-vector space'', ``linear'' refers to ``$\k$-linear'' and similarly for that of ``bilinear'' or ``multilinear''. All tensor products $\otimes$ without decoration are assumed to be over $\k$, and similarly for that of  symmetric and exterior  products.

The letter $\g$ always denotes a finite dimensional (nongraded) Lie algebra over $\k$. The Lie bracket of $\g$ is written as $[-,-]_\g$.
Given a $\g$-module $E$, the $\g$-action on $E$ is denoted by $\triangleright :~\g\otimes E\to E$. The Chevalley-Eilenberg complex
$\Omegag( E)=\wedge^\bullet \g^\vee\otimes E$
is endowed with the standard differential $d^E_{\CE}$ defined by:
\[
(d^E_{\CE}e)(x_1)=x_1\triangleright e,
\]
for $e\in E$, $x_1\in \g$, and
\begin{align*}
&\quad (d^E_{\CE} \omega) (x_1, \ldots, x_{n+1}) \\
 &= \sum_{i=1}^{n+1} (-1)^{i+1}  x_i \triangleright \omega (x_1, \ldots, \widehat{x_i}, \ldots, x_{n+1})  + \sum_{i <j} (-1)^{i+j} \omega([x_i, x_j], x_1, \ldots, \widehat{x_i}, \ldots, \widehat{x_j}, \ldots, x_{n+1}),
\end{align*}
for $\omega\in \wedge^n \g^\vee\otimes E$ ($n\geqslant  1$) and $x_i\in \g$.
In particular, $\Omegag :=\Omegag( \mathbb{K})=\wedge^\bullet \g^\vee$ together with $\dCE $
is a dg algebra where $\dCE \colon \mathbb{K}\to \g^\vee$ is zero and for $\omega\in \wedge^n \g^\vee\otimes E$ ($n\geqslant  1$), we have
\begin{align*}
 (\dCE  \omega) (x_1, \ldots, x_{n+1}) = \sum_{i <  j } (-1)^{i+j} \omega([x_i, x_j], x_1, \ldots, \widehat{x_i}, \ldots, \widehat{x_j}, \ldots, x_{n+1}).
\end{align*}
The cohomology of the dg algebra $(\Omegag ,\dCE )$ is denoted by $H^\bullet_{\CE}(\g)$.

\subsection{Dg modules over Lie algebras}~\label{Sec:dgrepLiealgebras}
Let $\g$ be a Lie algebra. We consider the objects of $\g$-modules in the category of bounded cochain complexes\footnote{In this note, we only work with bounded cochain complexes. In many cases, we also assume the cochain complex is non-negative.}.
The following definition is standard.

\begin{Def}~\label{defdgrep}
By a \textbf{dg $\g$-module}, or a \textbf{dg representation} of $\g$, we mean a representation of the Lie algebra $\g$ on a bounded cochain complex $V$
\[
  0 \to	V^{\mathrm{bot}}\xrightarrow{d^V_{\mathrm{bot}}}\cdots  \xrightarrow{d^V_{i-1}} V^{i} \xrightarrow{d^V_i}V^{i+1} \xrightarrow{d^V_{i+1}}\cdots \xrightarrow{d^V_{\mathrm{top}-1}} V^{\mathrm{top}}\to 0 .
\]
We require each $V^i$ to be a $\g$-module, and the $\g$-action is compatible with the differential $d^V_i$
\begin{align*}
  d^V_i(x \triangleright v)=x \triangleright d^V_i(v), \qquad\forall x\in \g, v\in V^i.
\end{align*}
\end{Def}

The complex $V$ is said to be \textbf{non-negative}, if it initiates from $V^{\mathrm{bot}=0}$:
\[
 0\to V^0\xrightarrow{d^V_0} V^{1} \xrightarrow{d^V_1}V^{2} \xrightarrow{d^V_2}\cdots \xrightarrow{d^V_{\mathrm{top}-1}} V^{\mathrm{top}}\to 0 .
 \]
For simplicity, we usually denote $d^V_i$ by $d^V$ if there is no risk of confusion. We also use a single letter $V$ to refer to the dg $\g$-module
$V= (V^\bullet, d^V)$ as defined above.

\begin{Ex}
For the standard Chevalley-Eilenberg cochain complex
\[
\mathbb{R}\xrightarrow{d^V_0=0} \g^\vee \xrightarrow{d^V_1= \dCE  }\wedge^2 \g^\vee \xrightarrow{d^V_2= \dCE  } \wedge^3 \g^\vee\xrightarrow{\ \ \ \ }\cdots
\]
associated to a Lie algebra $\g$,
each $V^i=\wedge^i \g^\vee$ is a $\g$-module via the Lie derivation $L_x:~V^i\to V^i $ ($x\in \g$), and the $\g$-action is compatible with $d^V_i$:
\begin{align*}
	d^V_i(L_x(\xi))=L_x (d^V_i(\xi)), \qquad\forall x\in \g, \xi\in \wedge^i\g^\vee.
\end{align*}
\end{Ex}

It is clear that the cohomology $\H(V)=\ker (d^V)/ \img (d^V)$ of the cochain complex $V$ inherits a $\g$-module structure. In particular, the zeroth cohomology $\H^{0}(V)=\ker (d_0^V)/\img (d_{-1}^V)$ is an ordinary  $\g$-module.
\begin{Rem}
Definition \ref{defdgrep} is a particular instance of the notion of representations up to homotopy of Lie algebroids on a cochain complex of vector bundles introduced  in \cites{AC1,AC} by Arias Abad and Crainic.
\end{Rem}

\begin{Def}
A dg $\g$-module $V=(V^\bullet,d^V)$ is called a \textbf{dg lift} of a given $\g$-module $E$, if
\begin{itemize}
\item $V$ is non-negative; and
\item $E$ is isomorphic to $\H^0(V)=\ker (d_0^V)$ as a $\g$-module.
\end{itemize}
We call $(V^\bullet,d^V)$ a \textbf{resolution} of the $\g$-module $E$, if it is a dg lift of $E$ and the cochain complex  $(V^\bullet,d^V)$ is \textbf{acyclic} in all positive degrees, i.e., $\H^i(V)=0$  for all $i \geqslant  1$.
\end{Def}

Each dg $\g$-module $V=(V^\bullet,d^V)$ induces a dg module over the dg algebra $(\Omega_\g, d_{\CE})$:  the underlying graded  $\Omegag $-module is given by
\[
\Omegag(V):=(\oplus_{p\geqslant  0} \wedge^p \g^\vee)\otimes (\oplus_{q} V^q).
\]
Elements in $\wedge^p\g^\vee\otimes V^q$ have degree $p+q$.
The coboundary map $d^V_i\colon V^i\to V^{i+1}$ is extended to a map $\Omegag( V^i)\to \Omegag( V^{i+1})$, by setting
\[
d^V_i(\omega \otimes v)=(-1)^p\omega\otimes d^V_i(v), \quad \omega\in \wedge^p \g^\vee,v\in V^i.
\]
The summation $d^V=\sum_{i}d^V_i$ is also regarded as an $\Omegag $-linear map $\Omegag( V)\to \Omegag( V)$ of degree $(+1)$.
It is easy to check the anti-commutative relation:
\begin{equation}~\label{atc}
d^V \circ \dCE ^V +\dCE ^V \circ d^V=0,
\end{equation}
where the horizontal differential
\[
\dCE ^V \colon \wedge^\bullet \g^\vee\otimes (\oplus_iV^i) \rightarrow \wedge^{\bullet+1} \g^\vee \otimes (\oplus_i V^{i})
\]
is the Chevalley-Eilenberg differential with respect to the $\g$-module $\oplus_i V^i$. Let us denote by
\[
 d_\tot^V=\dCE ^V+d^V\colon  \quad \Omegag( V)\to \Omegag( V)
 \]
the \textbf{total differential}, which satisfies $d^V_\tot \circ d^V_\tot=0$.
So we obtain a dg $\Omegag $-module $(\Omegag(V), d^V_\tot)$, which we call the
\textbf{total Chevalley-Eilenberg complex}. The corresponding cohomology is denoted by $H^\bullet_{\tot}(\g,V)$, which is a  $H^\bullet_{\CE}(\g) $-module.

One notices that $\H^{0}(V)=\ker (d_0^V)/\img (d_{-1}^V)$ is not isomorphic to $H^0_{\tot}(\g,V)$ unless the $\g$-action on $V$ is trivial.
\begin{Def}
Let $V=(V^\bullet,d^V)$ and $W=(W^\bullet,d^W)$ be two dg $\g$-modules.
A \textbf{weak morphism} of dg $\g$-modules from $V$ to $W$ is a degree $0$ and $\Omegag $-linear map $f \colon \Omegag(V) \to \Omegag( W)$ which intertwines the two relevant total differentials, i.e.,
\begin{equation}\label{mor-dg-general}
f \circ d^V_\tot  =d^W_\tot \circ f \colon \Omegag(V) \to \Omegag(W).
\end{equation}
The composition of two weak morphisms is defined in an obvious way.
\end{Def}

In the rest of this paper, a weak morphism as above is also denoted by $f\colon V \rightsquigarrow W$, though it should \emph{not} be understood as a linear map from $V$ to $W$. Indeed, the honest map is $f \colon \Omegag( V)\to \Omegag( W)$. Since $f$ is $\Omegag $-linear, it is generated by a family of  $\k$-linear maps $f_k \colon V^{k }\to \oplus_{l+s= k} \wedge^{l}\g^\vee \otimes W^{s}$. The index $k$ ranges from the lowest degree `$\mathrm{bot}$' to the highest degree `$\mathrm{top}$' of $V$.
Moreover, we can write $f_k$ as a summation of its summands:
\begin{equation}\label{Eqt:tempfk}
f_k=\sum_{0\leqslant  l \leqslant  \dim\g }f_k^l, \quad \mbox{where}\quad f_k^l \colon V^k\to \wedge^l\g^\vee \otimes W^{k-l}.
\end{equation}
A weak morphism $f\colon V \rightsquigarrow W$ is said to be  \textbf{strict} if it is generated by a family of morphisms of $\g$-modules, that is,  $\k$-linear maps $f_k \colon V^{k }\to W^{k}$ that are compatible with the $\g$-module structures on $V^k$ and $W^k$.

The map $f_k \colon V^{k }\to \oplus_{l+s= k} \wedge^{l}\g^\vee \otimes W^{s}$ as in Equation \eqref{Eqt:tempfk} is naturally extended to an $\Omegag $-linear map $f_k \colon \Omegag(V) \to \Omegag(W)$ of degree $0$ by defining
\[
f_k(\omega \otimes v^k) = \omega \wedge f_k(v^k),\quad \omega \in \wedge^\bullet \g^\vee, v^k\in V^k,
\]
and being trivial on other kinds of arguments:
\[
f_k(\omega\otimes v^j)=0, \quad j\neq k .
\]
Condition \eqref{mor-dg-general} now reads
\begin{equation}~\label{mor-dg}
f_{k+1}\circ d^V_k+f_k\circ \dCE ^V=d^W_\tot \circ f_k \colon V^k\to \oplus_{l+s=k+1}\wedge^l \g^\vee\otimes W^s,
\end{equation}
for each index $k$.

In particular, by taking $k=-1$ and $k=0$ respectively, Equation \eqref{mor-dg} yields the following relations:
\begin{eqnarray}
f_0^0\circ d^V_{-1}&=&d^W_{-1}\circ  f_{-1}^0,\quad \mbox{as a map}\quad V^{-1}\to W^0,~\label{mor-dg-0-0a}\\
f_{-1}^0\circ d^V_\CE+f_0^1\circ d^V_{-1}&=&d^W_\CE\circ f_{-1}^0+ d^W_{-2}\circ f^1_{-1}, \quad \mbox{as a map}\quad V^{-1}\to \g^\vee \otimes W^{-1},~\label{mor-dg-0-1b}\\
f_1^0\circ d^V_0&=&d^W_0\circ  f_0^0,\quad \mbox{as a map}\quad V^{0}\to W^1,~\label{mor-dg-0-0}\\
f_0^0\circ d^V_\CE+f_1^1\circ d^V_0&=&d^W_\CE\circ f_0^0+ d^W_{-1}\circ f_0^{1}, \quad \mbox{as a map}\quad V^0\to \g^\vee \otimes W^0.~\label{mor-dg-0-1}
\end{eqnarray}

Consider the $\g$-modules $\H^0(V)=\ker (d^V_0)/\img  (d^V_{-1})$ and
$\H^0(W)=\ker (d^W_0)/\img  (d^W_{-1})$.
By Equations \eqref{mor-dg-0-0a} and \eqref{mor-dg-0-0}, $ f_0^0$ induces a map from $\H^0(V)$ to $\H^0(W)$, and we denote it by $\H^0(f)$. From Equations \eqref{mor-dg-0-1b} and \eqref{mor-dg-0-1}, we see that $\H^0(f)$ is a morphism of $\g$-modules.

These facts boil down to a functor $\H^0$ from the category of dg $\g$-modules to the category of $\g$-modules: a dg $\g$-module $V$ is sent to $\H^0(V)$,
while a weak morphism of dg $\g$-modules $f \colon V\rightsquigarrow W$ is sent to $\H^0(f) \colon \H^0(V)\to \H^0(W)$.

\begin{Def}
Let $V$ and $W$ be dg lifts of the $\g$-modules $\H^0(V)$ and $\H^0(W)$, respectively. Let $\varphi \colon \H^0(V)\to \H^0(W)$ be a map of $\g$-modules. A \textbf{dg lift} of $\varphi $
is a weak morphism of dg $\g$-modules $f\colon V \rightsquigarrow  W$ such that $\H^0(f)=\varphi $.
\end{Def}

\begin{Def}~\label{Def:homotopy}
Let $f$ and $f' \colon V \rightsquigarrow W$ be weak morphisms of dg $\g$-modules. We say that $f$ and $f'$ are \textbf{homotopic}, if there exists a \textbf{homotopy} between $f$ and $f'$, i.e., a degree $(-1)$ and $\Omegag $-linear map $h\colon \Omegag( V)\to \Omegag( W)$ such that
\[
f^\prime-f =   d^W_{\tot} \circ h + h \circ d^V_{\tot} \colon  \Omegag(V)\to \Omegag(W).
\]
\end{Def}	

We will write $[f]$ for the \textbf{homotopy class} of the weak morphism  $f\colon V\rightsquigarrow W$. It is evident that two homotopic weak morphisms are sent by $\H^0$ to the same morphism of $\g$-modules. So it is natural to consider the reverse problem, i.e., the {existence and uniqueness of dg lifts of ordinary Lie algebra module morphisms}. The following theorem gives a sufficient condition, whose proof is given in the appendix.
\begin{Thm}~\label{Thm:basicdglift}
Let $V$ and $W$ be two non-negative dg $\g$-modules. Assume that $V$ and $W$ are dg lifts of $\g$-modules $E$ and $F$, respectively. If $V$ is a resolution of $E$ (i.e., $V$ is acyclic), then to every $\g$-module morphism $\varphi\colon E \to F$, there exists a weak morphism $f \colon V \rightsquigarrow W$ of dg $\g$-modules which is a dg lift of $\varphi$. Moreover, any two dg lifts of $\varphi$ are homotopic, i.e., the homotopy class $[f]$ is unique.
\end{Thm}

\subsection{{\LP modules}}~\label{Sec:classicalLPmodules}
Our motivation comes from what we call \LP modules introduced by Loday and Pirashvili in ~\cite{LP}. Roughly speaking, they are Lie algebras of linear maps. Let us first recall the definition of the category $\mathcal{LM}$ of linear maps.
% which is now widely referred to as the \LP category.
\begin{Def}\cite{LP}~
The objects in the category $\mathcal{LM}$ are $\k$-linear maps  $(V \overset{f}{\rightarrow} W) $, where $V$ and $W$ are (ordinary ungraded) $\k$-vector spaces. The morphisms  between objects in $\mathcal{LM}$ are pairs of $\k$-linear maps $\phi:= (\phi_1, \phi_0)$ such that the following diagram commutes
 \[
 \begin{tikzcd}
  V \arrow[r, "\phi_1"]    \arrow[d, "f"] & V^\prime  \arrow[d, "f^\prime"] \\
        W  \arrow[r, "\phi_0"]  &   W^\prime.
\end{tikzcd}
\]
Given two morphisms $\psi := ( \psi_1 , \psi_0)$ and $\phi := ( \phi_1 , \phi_0)$ in $\mathcal{LM}$, their composition $\phi \circ \psi$ is defined in the common way:
\[
\phi \circ \psi = ( \phi_1 , \phi_0 ) \circ ( \psi_1 , \psi_0 ) := ( \phi_1 \circ \psi_1, \phi_0 \circ \psi_0 ).
\]
\end{Def}

The category $\mathcal{LM}$ is in fact a tensor category\footnote{Besides the tensor product, one can define an   inner hom, equipping the category $\mathcal{LM}$ with a structure of a $\k$-linear closed symmetric monoidal category \cite{LP}.} where the tensor product of two objects is given by
\[
\bigl(V \xrightarrow{f}  W\bigr) \otimes\bigl(V^{\prime} \xrightarrow{f^{\prime}}  W^{\prime}\bigr) := \bigl(V \otimes W^{\prime} \oplus W \otimes V^{\prime} \xrightarrow{f \otimes 1_{W^{\prime}}+1_{W} \otimes f^{\prime}}  W \otimes W^{\prime}\bigr).
\]
We also have the standard \textit{interchange map}: for two linear maps $ (V \xrightarrow{f}  W )$ and $ (V^{\prime} \xrightarrow{f^{\prime}}  W^{\prime} )$,
a morphism
\[
\tau=\tau_{f,f'} \colon \bigl(V \xrightarrow{f}  W\bigr) \otimes\bigl(V^{\prime} \xrightarrow{f^{\prime}}  W^{\prime}\bigr) \to  \bigl(V^{\prime} \xrightarrow{f^{\prime}}  W^{\prime}\bigr)\otimes \bigl(V \xrightarrow{f}  W\bigr)
\]
is given by $\tau_1 \colon v \otimes w^{\prime} \oplus w \otimes v^{\prime}\mapsto  w^{\prime}\otimes v \oplus   v^{\prime}\otimes w$ followed by $\tau_0\colon w \otimes w^{\prime}  \mapsto  w^{\prime}\otimes   w$.

\begin{Def}\cite{LP}
A Loday-Pirashvili module is a Lie algebra object in the category $\mathcal{LM}$. In other words, 	
a Lie algebra in  $\mathcal{LM}$ is an object $\bigl(G\xrightarrow{X} \g \bigr)$ equipped with a morphism
\[
\mu \colon \bigl(G\xrightarrow{X} \g \bigr)\otimes\bigl(G\xrightarrow{X} \g \bigr)\longrightarrow\bigl(G\xrightarrow{X} \g \bigr)
\]
satisfying
\begin{itemize}
\item[(i)]  $\mu \circ \tau=-\mu$  and
\item[(ii)]   $\mu(1 \otimes \mu)-\mu(\mu \otimes 1)+\mu(\mu \otimes 1)(1 \otimes \tau)=0$.
\end{itemize}
\end{Def}
Equivalently, a Loday-Pirashvili module consists of a Lie algebra $\g$, a $\g$-module $G$, and a linear map $X\colon G \to \g$ which satisfies the $\g$-equivariant property (see \cite{LP}):
\[
X(x  \triangleright g) = [x ,X(g)]_{\g}, \qquad \forall x  \in \g, g \in G .
\]
In the sequel to this note, a \LP module as above will be denoted by $\bigl(G\xrightarrow{X} \g \bigr)$. Or, when the Lie algebra $\g$ is fixed, we simply refer to the pair $(G,X)$, or just $G$, as a \textbf{\LP module} (over $\g$). In other words, a \LP module structure is just a $\g$-module equipped with a morphism $X\colon G\to \g$ of $\g$-modules, where $\g$ is naturally treated as a $\g$-module by the adjoint action.

A morphism of {\LP modules} from $(G,X)$ to $(H,Y)$ is a morphism $\varphi\colon G\to H$ of $\g$-modules such that the following triangle commutes:
\[
\begin{tikzcd}
	{G} \ar{rr}{\varphi} \ar{dr}[left]{X} & &  {H} \ar{dl}{Y} \\
	&   \g. &
\end{tikzcd}
\]
The category of {\LP modules} over $\g$, denoted by $\mathcal{LP}(\g)$, consists of {\LP modules} and their morphisms as explained above.

\subsection{{Dg \LP modules}}

From now on, we only consider \textit{non-negative} dg $\g$-modules  $V=(V^\bullet, d^V)$.
In particular, one has the 1-term dg  $\g$-module $\g=\g[0]$, i.e., the trivial cochain complex which concentrates in degree $0$, and endowed with the standard adjoint action of $\g$ on $\g$. The corresponding (total) complex is denoted by $(\Omegag(\g), \dCE^{\g })$ (which has trivial $d^{\g }$ differential).

We are in a position to give the key notion of dg \LP modules, which is a generalization of \LP modules.
\begin{Def}~\label{Def: het0}
Let $\g$ be a Lie algebra.
A \textbf{dg \LP module} over $\g$ is a \textit{non-negative} and bounded dg $\g$-module $V=(V^\bullet, d^V)$ together with a weak morphism of dg $\g$-modules  $\alpha \colon V\rightsquigarrow \g$.
In other words, $\alpha$ is a degree $0$ morphism of $\Omegag $-modules
$\Omegag( V )\to  \Omegag( \g )$ which intertwines the relevant dg structures, i.e.,
\begin{equation}~\label{0 equation of f}
\alpha \circ d^V_\tot  =\dCE ^{\g }\circ \alpha.
\end{equation}
Such a dg \LP module is denoted by a pair $(V, \alpha )$.
\end{Def}
As a degree $0$ and $\Omegag $-linear map, each dg Loday-Pirashvili module $\alpha \colon \Omegag( V )\to  \Omegag( \g )$ is generated by a family of linear maps $\alpha_k:~V^{k }\to \wedge^{k}\g^\vee \otimes \g$. The index $k$ ranges from  $0$ to the top degree $\mathrm{top}(V)$ of $V$.  Clearly, $\alpha_k$ is trivial unless $0 \leqslant k \leqslant  \dim\g$.
In the sequel, let us denote
       \begin{equation}\label{Eqt:u}
        u=\mathrm{min}\{\mathrm{top}(V),\dim\g\},
        \end{equation}
and hence for $\alpha_k$ we only need to consider those with the index $0 \leqslant    k\leqslant  u$. Out of this range $\alpha_k$ is set to be zero by default.
Each $\alpha_k$ can be also considered as an $\Omegag $-linear map $\alpha_k:\Omegag( V)\to \Omegag( \g)$ of degree $0$ by setting
\[
		\alpha_k(\omega \otimes v^k)=\omega \wedge \alpha_k(v^k)~\mbox{and } \alpha_k(\omega \otimes v^j)=0~\mbox{for }j\neq k.
\]
In doing so,  condition \eqref{0 equation of f} can be reformulated in terms of  $ \alpha_k $ as follows:
\begin{equation}~\label{first equation of f}
\alpha_{k+1}\circ d^V_k+ \alpha_k\circ \dCE ^V=\dCE ^{\g }\circ \alpha_k,\quad\mbox{as a map }V^{k}\to \wedge^{k+1}\g^\vee\otimes \g,
\end{equation}
for all $0\leqslant  k\leqslant  u$.
In particular, if $V$ is concentrated in degree $0$, i.e., $V=V^0$ and $d^V=0$, then a dg \LP module $(V,\alpha)$ reduces to an ordinary \LP module $(V^0,\alpha)$.
\begin{Def}	
A \textbf{dg \LP class} is a pair $(V,[\alpha])$, where $ (V, \alpha)$ is a dg \LP module over $\g$ and $[\alpha]$ is the homotopy class of $\alpha \colon V\rightsquigarrow \g$.
\end{Def}
Two dg \LP  classes $(V,[\alpha])$ and $(V,[\alpha'])$ coincide if $\alpha$ and $\alpha'$ are homotopic, i.e., there exists a family of linear maps
\[
h=\{h_k \colon V^k\to \wedge^{k-1}\g^\vee\otimes \g\}_{1\leqslant  k\leqslant  u+1}
\]
such that for all $0\leqslant  k\leqslant  u$,  one has
\begin{equation}~\label{Eqt:alphaalphaprimehomotopic}
\alpha_k'-\alpha_k=\dCE ^{\g }\circ h_k+h_k\circ \dCE ^V+h_{k+1}\circ d^V_k,~\mbox{as a map }~V^k\to \wedge^k\g^\vee\otimes \g.
\end{equation}

We now define morphisms between dg \LP modules.
\begin{Def}~\label{Def:morphismweakLPmodules}
Let $(V, \alpha )$ and $(W, \beta)$ be two {dg \LP modules}. A morphism $ f $ from $(V, \alpha)$ to $(W, \beta)$ is a weak morphism $f\colon V \rightsquigarrow W$ of dg $\g$-modules such that $\beta \circ f= \alpha \colon V \rightsquigarrow \g$.
%\[
%\begin{tikzcd}
%		V  \arrow[rr, squiggly, "f"]  \arrow[dr, squiggly, "\alpha"] & &  W  \arrow[dl, squiggly, "\beta"] \\
%		&  \g &
%\end{tikzcd}
%\]
%is commutative in the category of dg $\g$-modules.	
%\end{Def}
%Moreover, we introduce the homotopy class of a weak morphism between two dg \LP classes as follows.
%\begin{Def}

Furthermore, a morphism $[f]\colon (V,[\alpha])\to (W,[\beta])$ is the homotopy class of a weak morphism $f\colon V\rightsquigarrow  W $ of dg $\g$-modules  such that the triangle
\[
\begin{tikzcd}
		V  \arrow[rr, squiggly, "f"] \arrow[dr, squiggly, "\alpha"] & &  W  \ar[dl, squiggly, "\beta"] \\
		&  \g &
\end{tikzcd}
\]
is commutative \emph{up to homotopy} in the category of dg $\g$-modules.
\end{Def}
It follows from direct verification that the collection of {dg \LP modules} over $\g$ together with their morphisms forms a category which we denote by $\mathcal{WLP}(\g)$.
By $[\mathcal{WLP}](\g)$ we denote the category of dg \LP classes over $\g$ together with their morphisms.
Many statements about dg \LP modules in the sequel are indeed homotopy invariant, and thus holds at the level of dg \LP classes.

\subsection{Alternative descriptions of dg \LP modules}
Let $(V, \alpha )$ be a {dg \LP module} over $\g$.
One can treat each summand $\alpha_k \colon V^{k } \to \wedge^{k}\g^\vee \otimes \g$ of $\alpha$ as a linear map $V^{k } \otimes \wedge^{k} \g  \rightarrow \g, 0\leqslant  k\leqslant  u$, where $u$ is defined by \eqref{Eqt:u}.
The condition \eqref{first equation of f} for $\alpha$ being a weak morphism of dg $\g$-modules can be reformulated from a different perspective.

\begin{prop}~\label{Prop:otherdescriptionsofLP}
Given a family of linear maps $\{\alpha_k: V^{k } \otimes \wedge^{k} \g  \rightarrow \g\}_{0\leqslant  k\leqslant  u}$, the corresponding pair $(V, \alpha )$ with
$\alpha\colon \Omegag( V )\to  \Omegag( \g )$ defines a {dg \LP module} if and only if the following equations hold for $0\leqslant  k\leqslant  u$:
\begin{eqnarray}~\label{equation of f}
\notag &&\alpha_{k+1}\big(d^{V}_kv \mid x_{1}, \cdots, x_{k+1}\big)\\
\notag &=&\sum_{i=1}^{k+1}(-1)^{i+1}\big(\left[x_{i}, \alpha_{k} (v \mid x_{1}, \cdots, \widehat{x_{i}}, \cdots, x_{k+1} )\right]-\alpha_{k} (x_{i} \triangleright v \mid x_{1},
    \cdots, \widehat{x_{i}},\cdots, x_{k+1} )\big)\\
&&+\sum_{i<j}(-1)^{i+j} \alpha_{k}\bigl(v \mid \left[x_{i}, x_{j}\right], x_{1}, \cdots, \widehat{x_{i}}, \cdots, \widehat{x_{j}}, \cdots, x_{k+1}\bigr),\\\notag
&&
\end{eqnarray}
for all $x_i \in \g$ and $v \in V^{k }$.
\end{prop}

The proof is omitted here because it is a direct verification using the definitions of $d^V_k$, $\dCE ^V$, and $\dCE ^{\g }$. From now on, a {dg \LP module} $(V, \alpha)$ could be alternatively denoted by
\[
\alpha=\{\alpha_k\colon V^{k } \to  \wedge^{k} \g^\vee  \otimes \g \}_{0\leqslant  k\leqslant  u},
\]
or
\[
\{ \alpha_k \colon V^{k} \otimes \wedge^{k} \g  \rightarrow \g \}_{0\leqslant  k\leqslant  u},
\]
or simply $ \alpha =\{\alpha_k\} $.
For the case that $k=0$ and $k=1$, Equation~\eqref{equation of f} is now given as follows:
\begin{compactenum}
\item The first equation reads
	\[
	[x,\alpha_0(v)]_\g-\alpha_0(x \triangleright v) = \alpha_{1}(d^V_0 v  \mid x),\;\;\forall x \in \g, v \in V^{0}.
\]
So the component $\alpha_1$ characterizes the failure of $\alpha_0: V^{0} \to \g$ being a \LP module in the sense of \cite{LP}.
\item The second equation reads
\begin{align*}%~\label{eqn:f1}
	\alpha_{2}(d^V_1 v  \mid x_{1}, x_{2}) &= \alpha_{1}\bigl(x_{2}\triangleright v \mid x_{1}\bigr)-\alpha_{1}\bigl(x_{1} \triangleright v \mid x_{2}\bigr)-\alpha_1(v \mid
      [x_1,x_2]) \notag \\
	&\quad +[x_{1}, \alpha_{1}\bigl(v \mid x_{2}\bigr) ]-\left[x_{2}, \alpha_{1}\bigl(v \mid x_{1}\bigr)\right],
\end{align*}
for all $x_1,x_2 \in \g, v \in V^1$. Thus, the component $\alpha_2$ characterizes the failure of $\alpha_1 \colon V^1 \to \g^\vee \otimes \g$ being a morphism of $\g$-modules.
\end{compactenum}

\begin{Ex}~\label{Ex: pairing}
Let $\g$ be a Lie algebra equipped with an invariant pairing $\langle -,- \rangle$. Consider the vector space $V=V^1:=\g  \otimes \g$ concentrated in degree $(+1)$ and
endowed with the obvious adjoint representation of $\g$. The invariant pairing induces an equivariant linear map
\[
  \g  \to \g ^\vee,\quad x \mapsto x^\sharp := \langle x,-\rangle.
\]
It also induces a {dg \LP module} by setting
\[
  \alpha=\{\alpha_1 \colon \g  \otimes \g  \to \g^\vee \otimes \g\} ,\qquad \alpha_1(x \otimes y) = x^\sharp \otimes y ,\;\;
  \forall x , y \in \g.
\]
\end{Ex}

\begin{prop}~\label{Prop:homotopyinh}
Two dg \LP modules $(V, \alpha)$ and $(V, \alpha')$ are homotopic if and only if there exists a family of maps
\[
 h=\{h_k\colon V^k. \otimes \wedge^{k-1}\g \to   \g\}_{1\leqslant  k\leqslant  u+1}
 \]
such that the following equality holds for all $v \in V^k$ and $x_1,\cdots, x_k\in \g$:
\begin{eqnarray*}
&&(\alpha_k'-\alpha_k)(v \mid x_1,\cdots,x_k)\\
&=& \sum_{i=1}^{k} (-1)^{i+1}  [ x_i, h_k (v\mid x_1, \ldots, \widehat{x_i}, \ldots, x_{k})]  + \sum_{i <  j} (-1)^{i+j} h_k(v\mid [x_i, x_j], x_1, \ldots,
     \widehat{x_i}, \ldots, \widehat{x_j}, \ldots, x_{k})\\
& &\qquad + \sum_{i=1}^{k} (-1)^{i}  h_k(x_i \triangleright v\mid x_1, \ldots, \widehat{x_i}, \ldots, x_{k})   +h_{k+1}(  d^V_kv\mid x_1,\cdots,x_k).
\end{eqnarray*}
\end{prop}
In fact, the above relation is the unravelled form of Equation \eqref{Eqt:alphaalphaprimehomotopic}.

We finally state another characterization of dg \LP modules. Note that a set of multilinear maps
$\alpha = \{\alpha_k\colon V^{k } \to \wedge^k \g^\vee  \otimes \g \}_{0\leqslant  k\leqslant  u} $
can be reorganized as a degree $0$ cochain
\[
c(\alpha) =  c_0(\alpha)+c_1(\alpha)+\cdots+c_{u}(\alpha) \in C (\g , \Hom(V, \g) ),
\]
where
\[
c_k(\alpha) \in \wedge^k \g^\vee\otimes \Hom(V^{k }, \g)
\]
is the equivalent form of $\alpha_k$.
It follows from a straightforward verification that the set $\{\alpha_k\} $ satisfies Equation~\eqref{first equation of f} if and only if $ c(\alpha) $ is a $0$-cocycle.
Based on these facts, one can  prove the following proposition.

\begin{prop}
With notations as above, there is a one-to-one correspondence between dg \LP modules $(V,\alpha)$ over $\g$ and degree $0$ cocycles $c(\alpha)$ of the Chevalley-Eilenberg total complex $ \Omegag(  \Hom(V,  \g) )$ of the Lie algebra $\g$ with coefficients in $\Hom(V,  \g)$. Moreover, there is a one-to-one correspondence between dg \LP modules classes  $(V,[\alpha])$ over $\g$ and degree $0$ cohomology classes of the Chevalley-Eilenberg total complex $ \Omegag(  \Hom(V,  \g) )$.
\end{prop}

\subsection{Dg Lifts of {\LP modules}}

In Section \ref{Sec:dgrepLiealgebras}, we have introduced the functor $ \H^0 $ from the category of dg modules to the category of ordinary $\g$-modules. Clearly, it also induces a functor from the category $[\mathcal{WLP}](\g)$ of dg \LP classes to the category $\mathcal{LP}(\g)$ of \LP modules.
In specific, a dg \LP class $(V,[\alpha])$ is sent by the functor $\H^0$ to the \LP module:
\[
\H^0(V)\xrightarrow{\H^0(\alpha)}{\g},
\]
where $\H^0(V)=\ker (V^{0}\xrightarrow{d_0^V} V^1)$
and $\H^0(\alpha)$ is the restriction of $\alpha_0\colon V^0\to \g$ on $\H^0(V)$.
This definition does not depend on the choice of $\alpha$ that represents the homotopy class $[\alpha]$.

Similarly, a morphism of dg \LP classes
$[f]\colon (V,[\alpha]) \rightsquigarrow (W,[\beta])$ is mapped to the morphism of {\LP modules}
\[
\H^0(f)\colon (\H^0(V),\H^0(\alpha))\to (\H^0(W),\H^0(\beta)),
\]
where $\H^0(f)=f_{0}^0|_{\H^0(V)} $. Note that $\H^0(f)$ is well-defined since it is independent of the choice of $f$ representing $[f]$.

\begin{Def}
Given a \LP module $(G,X)$, a \textbf{dg lift} of $(G,X)$ is a dg \LP module $(V, \alpha)$ such that $(\H^0(V), \H^0(\alpha)) $ is isomorphic to $(G,X)$ as \LP modules.
\end{Def}

We state and prove the first fundamental theorem of the paper.

\begin{Thm}~\label{Thm:foundamentalone}	
Let $G$ be a $\g$-module and $V=(V^\bullet, d^V)$ a resolution of $G$. To every {\LP module} $(G,X)$ over $\g$, there exists a dg lift $(V, \alpha)$ of $(G,X)$. Moreover, the dg lift is unique up to homotopy.
\end{Thm}
\begin{proof}  	
A {\LP module} $(G,X)$ is just a morphism of $\g$-modules $X\colon G\to \g$. By Theorem \ref{Thm:basicdglift}, we can find a dg lift of $X$, i.e., a weak morphism $\alpha\colon V \rightsquigarrow \g$ of dg $\g$-modules, where we regard $\g$ as a trivial complex concentrated in degree $0$ and a dg module of $\g$ in the obvious sense. The uniqueness (up to homotopy) of $\alpha$ is a consequence of Theorem \ref{Thm:basicdglift}.
\end{proof}

It is natural to consider dg lifts of morphisms of \LP modules. We give the definition in a standard fashion.
\begin{Def}
Let $(V, \alpha)$ and $(W, \beta)$ be, respectively, dg lifts of \LP modules $(G,X)$ and $(H,Y)$. Let $\varphi \colon (G,X) \to (H,Y)$  be a morphism of \LP modules.
A \textbf{dg lift} of $\varphi $ is a morphism of dg \LP classes $f \colon (V, \alpha) \rightsquigarrow (W, \beta)$ such that $\H^0(f)$ is compatible with $\varphi $, i.e., the following diagram commutes:
\[
   \begin{tikzcd}
		( \H^0(V),\H^0(\alpha) ) \ar{d}{\cong} \ar{r}{\H^0(f)} & (\H^0(W),\H^0(\beta) ) \ar{d}{\cong} \\
		( G,X) \ar{r}{\varphi} & (H,Y).
\end{tikzcd}
\]	
\end{Def}

The second fundamental theorem concerns the existence and uniqueness of dg lifts of morphisms of \LP modules.

\begin{Thm}~\label{Thm:acyclicgivesliftsofmorphism}
Let $(V, \alpha)$ and $(W, \beta)$ be the dg lifts of \LP modules $(G,X)$ and $(H,Y)$ respectively. 	
Suppose that $V$ is also a resolution of $G$. Then, to every morphism of {\LP modules} $\varphi \colon (G,X) \to (H,Y)$, there exists a dg lift of $\varphi$
\[
 f \colon  (V, \alpha) \rightsquigarrow (W, \beta),
\]
which is unique up to homotopy.
\end{Thm}

\begin{proof}
By Theorem \ref{Thm:basicdglift}, $\varphi$ lifts to a weak morphism of dg $\g$-modules $f \colon  V \rightsquigarrow W$, which is unique up to homotopy.
Since $\beta\circ f \colon V\rightsquigarrow  \g$ lifts $Y\circ \varphi \colon G \to \g$ which is identically $X$, we know that $\beta\circ f$ and $\alpha$ are necessary to be homotopic, i.e., $[\beta]\circ [f]=[\alpha]$, as desired. 	
\end{proof}

\section{Twisted Atiyah classes of dg \LP modules}~\label{Sec:twistedAtiyahclassesetc}
Atiyah introduced a cohomology class in~\cite{Atiyah}, which has become known as the Atiyah class, to characterize the obstruction to the existence of holomorphic connections on a holomorphic vector bundle. The notion of Atiyah classes has been extensively developed in the past decades for diverse purposes (see~\cites{Bottacin,CV,CLX,Costello,CSX,LaurentSX-CR,LaurentSX,MSX}). In this section, we first interpret each dg module $(V,\alpha)$ over $\g$ as a dg derivation of the commutative dg algebra $(\Omega_\g = \wedge \g^\vee, d_{\CE})$. Then we study the twisted Atiyah class of this dg derivation in the sense of~\cite{CLX}. It turns out that the twisted Atiyah class induces a Leibniz algebra structure on the ($(-1)$-shifted) total Chevalley-Eilenberg cohomology of the dg module $V$.

\subsection{Dg derivations}~\label{Sec:twistedAtiyahclassdgderivation}
A $\Z$-graded commutative dg algebra over $\k$ (simplified as dg algebra) is a pair $(\A,d_\A)$, where $\A$ is a commutative and $\Z$-graded $\k$-algebra,
and $d_\A \colon \A  \rightarrow \A$, usually called the differential, is a homogeneous degree $(+1)$ derivation which squares to zero.
By a {dg} $\A$-module, we mean an $\A$-module $\E$, together with a degree $(+1)$ and square zero endomorphism $\partial_\A^\E$ of the graded $\k$-vector space $\E$, also called the differential, such that
\[
\partial^\E_\A  (a e) = (d_\A a )  e + (-1)^{\abs{a}}a \partial_\A^\E(e)
\]
holds for all $a \in \A , e \in \E $.
To work with various different {dg} $\A$-modules, the differential $\partial_\A^\E$ of any {dg} $\A$-module $\E$ will be denoted by the same notation $\partial_\A$.

In this note, we assume that \emph{all $\A$-modules are projective}.
For example, if $\g$ is a Lie algebra, then the space $ \A=\Omegag = \wedge^\bullet \g^\vee$ together with $d_\A=\dCE $ is a dg algebra.
More generally, for a $\g$-module $E$, we have the standard Chevalley-Eilenberg differential $\partial_\A=d^E_{\CE}$ on $\E=\Omegag( E)=\wedge^\bullet \g^\vee\otimes E$, and $\E$ becomes a dg $\A=\Omegag $-module.

We recall the notion of {dg} module valued derivations ({dg} derivations for short) introduced in \cite{CLX}.
\begin{Def}~\label{Def:dgderivations}
Let $(\A,d_\A)$ be a dg algebra and $(\Omega,\partial_\A)$ a {dg} $\A$-module.
\begin{itemize}
\item A \textbf{{dg} derivation} of $\A$ valued in $(\Omega, \partial_\A)$ is a degree $0$ derivation $\delta \colon \A \rightarrow \Omega$ of the graded algebra $\A$ valued in the $\A$-module $\Omega$,
		\[
		\delta(aa') = \delta(a)a' + a\delta(a') = (-1)^{\abs{a}\abs{a'}} a'\delta(a)  + a\delta(a'),\;\;\;\forall a,a' \in \A ,
		\]
		which commutes with the differentials as well:
		\begin{equation}~\label{compatibility}
		\delta \circ d_\A = \partial_\A \circ \delta \colon \A  \rightarrow \Omega.
		\end{equation}
		Such a {dg} derivation is simply denoted by $\A \xrightarrow{\delta} \Omega$.
\item Let $\delta$ and $\delta^\prime$ be two $(\Omega,\partial_\A)$-valued {dg} derivations of $\A$. They are said to be \textbf{homotopic}, written as $\delta\sim \delta'$,
      if there exists a degree $(-1)$ derivation $h$ of $\A$ valued in the $\A$-module $\Omega$ such that
		\[
		\delta^\prime - \delta = [\partial_\A,h] = \partial_\A \circ h + h \circ d_\A\colon  \A \to \Omega.
		\]
	\end{itemize}
\end{Def}

A typical example of {dg} derivations arises from complex geometry --- Consider a complex manifold $M$. The Dolbeault dg algebra $\A=(\Omega_M^{0,\bullet},\bar{\partial})$ is a     dg algebra. Let $\Omega=(\Omega_M^{0,\bullet}((T^{1,0}M)^\vee), \bar{\partial})$ be the dg $\A$-module generated by the section space of holomorphic cotangent bundle $(T^{1,0}M)^\vee$. Then $\partial: \A \to \Omega$ is an $\Omega$-valued derivation of $\A$.
Examples of dg module valued derivations arise from dg Lie algebroids and Lie pairs can be found in \cite{CLX}.

Denote the $\A$-dual $\Omega^\vee$ of $\Omega$ by $\B$, which is also a {dg} $\A$-module. A {dg} derivation $\A \xrightarrow{\delta} \Omega$ can be alternatively described by covariant derivations along elements in $\B$:
\[
\nabla_b \colon \A \to \A ,\qquad \nabla_b(a) := \langle b|\delta(a)\rangle,\quad~\forall b\in \B, a\in \A.
\]
The above condition \eqref{compatibility} is also expressed by the relation
\[
d_\A (\nabla_b a)=\nabla_{\partial_\A b} ~a +(-1)^{|b|} \nabla_b(d_\A a).
\]

We need the category of dg derivations, denoted by $\CATderivationA$, whose objects are dg derivations $\A \xrightarrow{\delta} \Omega$ as in Definition \ref{Def:dgderivations}, and whose morphisms are defined as follows.
\begin{Def}~\label{Def:morphismDerA}
A morphism $\phi$ of dg derivations from $\A \xrightarrow{\delta} \Omega$ to $\A \xrightarrow{\eta} \Theta$ is a morphism $\phi\colon \Omega \rightarrow \Theta$ of dg $\A$-modules such that
\[
	\eta = \phi \circ \delta\colon \A \rightarrow \Theta.
\]
\end{Def}

\subsection{Characterization of dg \LP modules via dg derivations}
We first set up our notations. Let $\g$ be a Lie algebra. Consider the dg algebra $\A= \Omegag$ with its differential $d_\A = \dCE $.
As explained in previous sections, every dg  $\g$-module $W = (W^\bullet,d^W)$ gives rise to a dg $\Omegag $-module $(\Omegag( W ), d^W_\tot)$.

Next, we characterize dg \LP modules via dg derivations of the dg algebra $\Omegag$.
The following is an adaptation of the Definition \ref{Def:dgderivations} to dg $\Omegag $-modules.

\begin{Def}
Let $S = (S^\bullet,d^S)$ be a dg $\g$-module and $(\Omegag( S),d^S_\tot)$ the associated dg $\Omegag $-module.
\begin{itemize}
 \item An $\Omegag( S)$-valued dg derivation of $\Omegag  $ is a degree $0$ linear map
	$\delta\colon  \Omegag    \rightarrow  \Omegag(S)$
	that satisfies the Leibniz rule
	\[
	\delta(\xi \wedge \eta) = (-1)^{\abs{\xi}\abs{\eta}}\eta \wedge \delta(\xi) + \xi \wedge \delta(\eta),\;\;\forall \xi,\eta \in \Omegag
	\]
	and the compatibility condition
	\begin{equation}~\label{commute diagram}
	\delta \circ \dCE  = d^S_\tot \circ \delta \colon  \Omegag   \to \Omegag( S).
	\end{equation}
 \item Let $\delta$ and $\delta'$ be $\Omegag( S)$-valued dg derivations of $\Omegag  $. They are said to be  homotopic, written as $\delta\sim \delta'$, if there exists a degree $(-1)$ derivation $h$: $\Omegag  \to  \Omegag(S)$ such that
   \[
    \delta^\prime - \delta =  d^S_\tot   \circ h + h \circ \dCE  \colon   \Omegag  \to  \Omegag( S).
   \]
\end{itemize}
\end{Def}

Definition \ref{Def:morphismDerA} is also adapted:
A morphism $\phi$ of dg derivations from $\Omegag \xrightarrow{\delta} \Omegag(S)$ to $\Omegag   \xrightarrow{\eta} \Omegag(T)$ is a morphism $\phi\colon \Omegag(S) \rightarrow \Omegag(T)$ of dg $\Omegag$-modules such that
\[
	\eta = \phi \circ \delta \colon \Omegag   \rightarrow \Omegag( T).
\]

\begin{prop}~\label{prop1}
Let $V$ be a non-negative dg $\g$-module and $V^\vee$ endowed with the dg $\g$-module structure in duality to $V$.
A {dg \LP module} structure on $V$ is equivalent to a dg derivation of $\Omegag$ valued in  $\Omegag(V^\vee[-1])$.
Denote this correspondence by $(V, \alpha)\leftrightarrow (\Omegag\xrightarrow{\delta_\alpha}\Omegag(  V^\vee[-1]))$.
Moreover, the correspondence $f\leftrightarrow f^\vee$ is one-to-one, for morphisms $f$ of dg \LP modules from $(V, \alpha )$ to $(W, \beta )$ as in Definition \ref{Def:morphismweakLPmodules} and morphisms $f^\vee$ of dg derivations $f^\vee$ from $\Omegag \xrightarrow{\delta_\beta} \Omegag(  W^\vee[-1])$ to
$\Omegag \xrightarrow{\delta_\alpha}\Omegag(V^\vee[-1])$.
\end{prop}
\begin{proof}
We assume that $\delta_\alpha\colon \Omegag \rightarrow \Omegag( V^\vee[-1])$ is a derivation of $\Omegag $-modules. Suppose that it is a summation of derivations
\[
 \delta_\alpha=\delta_0 + \cdots + \delta_{u},
\]
where $u$ is defined by Equation \eqref{Eqt:u} and each summand is generated by
\[
\delta_k\!\mid_{\g ^\vee} \colon \g ^\vee \rightarrow \wedge^{k} \g^\vee  \otimes (V^{k})^\vee[-1].
\]
With the set $\{\delta_k\}$ we can associate a family of linear maps
  \[
	\alpha_k\colon V^{k} \to \wedge^{k} \g^\vee \otimes \g ,\;\;\;0\leqslant  k \leqslant  u,
  \]
such that
\begin{equation}~\label{deltaandalphak}
\langle  \alpha_k(v) | \xi \rangle = \langle v[1]|\delta_k(\xi) |   \rangle,\qquad \forall \xi\in \g^\vee, v\in V^k.
\end{equation}

Then the compatibility condition \eqref{commute diagram} of $\delta_\alpha$ being a dg derivation translates to condition \eqref{first equation of f}, i.e., $\alpha = \{\alpha_k\} $ being a dg \LP module. The relation between $\delta_{\alpha}$ and $\alpha$ is clearly one-to-one. The second statement follows easily from definitions of the relevant morphisms.
\end{proof}

 \begin{Cor}~\label{Cork=2}
Let $V = V^2$ be a $\g$-module concentrated in degree $2$. Then a linear map
\[
\alpha_2 \colon V^2 \to  \wedge^2 \g^\vee  \otimes \g
\]
defines a {dg \LP module} $\alpha = \{   \alpha_2\}$ if and only if the corresponding
\[
\alpha_2 \colon \wedge^2 \g\to \Hom(V^2,\g)
\]
defines an abelian extension of $\g$ along $\Hom(V^2, \g)$.
\end{Cor}

\begin{proof}
By (1) of Proposition \ref{prop1}, $\alpha= \{\alpha_2\}$ corresponds to a $0$-cocycle $c_\alpha \in C^2(\g, \Hom(V^2, \g))$. If we forget the degree of $\Hom(V^2, \g)$, then $c_\alpha$ is a $2$-cocycle.
Thus it determines an abelian extension of $\g$ along $\Hom(V^2, \g)$.
\end{proof}

Finally, we state a proposition without proof since it is easy to verify.
\begin{prop}~\label{Prop:homotopy-homotopy}
Two dg \LP modules $(V, \alpha)$ and $(V, \alpha')$ are homotopic if and only if they are homotopic as dg derivation of $\Omegag$ valued in $\Omegag(V^\vee[-1])$.
\end{prop}

\subsection{Twisted Atiyah classes of dg derivations}~\label{Sec:twistedAtiyahclassdgderivation2}
We need another key notion -- $\delta$-connections which is also introduced in \cite{CLX}. One can regard them as operations extending a given dg derivation $\delta$.
\begin{Def}~\label{Def-delta-connection}
Let $\A \xrightarrow{\delta} \Omega$ be a {dg} derivation and $\E$ an $\A$-module. A \textbf{$\delta$-connection} on $\E$ is a degree $0$ linear map of graded $\k$-vector spaces
  \[
	\nabla \colon \E  \rightarrow  \Omega  \otimes_\A \E
  \]
   satisfying the following Leibniz rule:
  \begin{equation}~\label{Leib of connections}
	\nabla(ae) = \delta(a) \otimes e + a\nabla e ,\;\;\forall a \in \A , e \in \E.
   \end{equation}
\end{Def}
A $\delta$-connection $\nabla\colon \E  \rightarrow \Omega \otimes_\A \E$ of the $\A$-module $\E$ is also expressed by the covariant derivation along elements $b \in \B$:
\[
 \nabla_b \colon E \to \E ,\qquad \nabla_b e := \langle b | \nabla e \rangle,\quad~\forall e\in \E.
 \]
Equation \eqref{Leib of connections} is indeed saying
\[
\nabla_b(ae)= (\nabla_b a)e + (-1)^{|a||b|} a \nabla_b e,\qquad \forall a\in \A.
\]

The following notion of twisted Atiyah class is introduced in \cite{CLX}.

\begin{prop-def}[Twisted Atiyah class]~\label{prop:Atiyah via connection}
Let $\A \xrightarrow{\delta} \Omega$ be a {dg} derivation and $\E = (\E, \partial_\A)$ a {dg} $\A$-module.
\begin{itemize}
	\item[1)] For any $\delta$-connection $\nabla $ on $\E$, the degree $(+1)$ element
		\[
		\At_\E^{\nabla}:=[\nabla,\partial_\A] = \nabla  \circ \partial_\A - \partial_\A \circ \nabla  \in \Omega \otimes_\A \End_\A(\E)
		\]
		is a cocycle.
	\item[ 2)] The cohomology class $[\At^{\nabla }_\E] \in H^1(\A, \Omega \otimes_\A \End_\A(\E))$ is independent of the choice of $\delta$-connections $\nabla$ on $\E$.
\end{itemize}
We call $\At_\E^{\nabla }$ the \textbf{$\delta$-twisted Atiyah cocycle} with respect to the $\delta$-connection $\nabla $ on $\E$, and $\At_\E^\delta:=[\At^{\nabla }_\E]$ the \textbf{$\delta$-twisted Atiyah class} of $\E$.
\end{prop-def}

The $\delta$-twisted Atiyah cocycle $\At_\E^{\nabla }$ could be viewed as a degree $(+1)$ element in $\Hom_\A(\B \otimes_\A \E,\E)$ by setting
\begin{align}~\label{Atiyah cocycle}
\At_\E^{\nabla }(b,e) &= (-1)^{\abs{b}}\langle b|\At_\E^{\nabla }(e) \rangle= (-1)^{\abs{b}}\langle b|\nabla (\partial_\A  (e)) - \partial_\A (\nabla  e )\rangle  \notag \\
&= -\partial_\A  (\nabla _b e) + \nabla _{\partial_\A (b)}e + (-1)^{\abs{b}}\nabla _b\partial_\A  (e),
\end{align}
for all $b \in \B $ and $e \in \E$.
Moreover, as $\At_\E^{\nabla }$ is a $1$-cocycle, it is a morphism of {dg} $\A$-modules, i.e., $\At_\E^{\nabla } \in \Hom^1_{\dg\A}(\B \otimes_\A \E, \E)$.

In \cite{CLX}, it is shown that twisted Atiyah classes are homotopy invariant.
If $\delta\sim\delta^\prime$, then for any {dg} $\A$-module $\E$, we have
\[
\At_\E^\delta = \At_\E^{\delta^\prime} \in H^1(\A, \Omega \otimes_\A \End_\A(\E))=H^1(\A, \Hom_\A(\B \otimes_\A \E,\E)).
\]
Moreover, the $\delta$-twisted Atiyah class $\At^\delta_\E$ vanishes if and only if there exists a $\delta$-connection $\nabla$ on $\E$ such that the associated twisted Atiyah cocycle $\At_\E^{\nabla }$ vanishes, i.e., the map $\nabla:~\E \rightarrow \Omega \otimes_\A \E$ is compatible with the differentials.
So $\At^\delta_\E$ can be thought of as the obstruction to the existence of such $\nabla$.

By a dg $\A$-algebra, we mean a commutative dg algebra $C$ which is endowed with an
dg $\A$-module structure. The following result is due to \cite{CLX}\footnote{The original  \cite{CLX}*{Corollary 3.15} is missing the consideration of taking some dg $\A$-algebra $C$ as a coefficient. In fact, if $C$ is chosen to be the naive case $C=\A$, the corresponding two structure maps ${\Bigl[-,-\Bigr]}_{\delta}$ and $-\triangleright-$ are both trivial. Therefore, we have revised the original statement here.}:
\begin{prop}~\label{maincoro}
Let $\A \xrightarrow{\delta} \Omega$ be a dg derivation, $\B$ be the dual dg $\A$-module  of $\Omega$, and $C$ a dg $\A$-algebra.
\begin{compactenum}
 \item The degree $(-1)$ shifted cohomology space $H^\bullet(\A,C\otimes \B[-1])$ is a Leibniz algebra over $H^\bullet(\A) $, whose bracket ${\Bigl[-,-\Bigr]}_{\delta}$ is induced by $C$-bilinear extension of the $\delta$-twisted Atiyah class of $\B$:
	\[
 \Bigl[[b_1 ],[b_2 ]\Bigr]_{\delta}= (-1)^{\abs{b_1}}\At^\delta_\B ([b_1],[b_2]),
    \]
	  where $b_1,b_2 \in C\otimes \B$ are $\partial_\A$-closed elements.
\item For any dg $\A$-module $\E$, there exists a representation of the Leibniz algebra $H^\bullet(\A,C\otimes \B[-1])$ on the cohomology space $H^\bullet(\A,C\otimes \E)$, with the action map $-\triangleright-$ induced by $C$-bilinear extension of the $\delta$-twisted Atiyah class of $\E$:
\[
 [b] \triangleright [e]= (-1)^{\abs{b}}\At^\delta_\E ([b],[e]),
 \]
 where $b\in C\otimes \B$, $e\in C\otimes \E$ are both $\partial_\A$-closed elements.
\end{compactenum}
\end{prop}

\subsection{Twisted Atiyah classes of dg \LP modules}
As illustrated by Proposition \ref{prop1}, a {dg \LP module} $(V, \alpha)$ is equivalent to a dg derivation
\[
\delta_\alpha\colon \Omegag   \rightarrow \Omega=\Omegag( V^\vee[-1])
\]
of $\A=\Omegag$ such that Equation \eqref{deltaandalphak} holds.
Suppose further that we are given another dg $\g$-module $W = (W^\bullet,d^W)$ which corresponds to the dg $\Omegag $-module $\E=(\Omegag( W ),d^W_\tot)$.
In the rest of this subsection, we figure out the $\delta_\alpha$-twisted Atiyah cocycle and class in this setting.

The dual module of $\Omega=\Omegag( V^\vee[-1])$ is $\B=\Omegag( V[1])$ (since we have assumed that $V$ is finite dimensional).
The dg derivation $\delta_\alpha$ yields the covariant derivation $\nabla_b \colon \Omegag \to \Omegag$ for all $b\in \B$. More precisely, for all $b = \eta \otimes v[1] \in \B$, where $\eta \in \Omegag, v[1] \in V[1]$, we have $\nabla_{b} = \nabla_{\eta \otimes v[1]} = \eta \otimes \nabla_{v[1]}$. Here the derivation $\nabla_{v[1]}$ is given by the following relation
\[
\nabla_{v[1]} \colon \Omegag \to \Omegag ,\mbox{ such that }
\nabla_{v[1]}\xi =  \langle v[1]|\delta_\alpha (\xi) |     \rangle=\langle  \alpha (v) | \xi \rangle,
\]
where $v\in V$ and $\xi\in \g^\vee$. In other words, we may regard
\[
\nabla_{v[1]} = \iota_{\alpha(v)}\colon \Omegag \to \Omegag
\]
as a derivation.

We then take the most naive $\delta_\alpha$-connection $\nabla$ on $\E$ such that
\begin{equation}~\label{Eqt:naivedeltaconnection}
\nabla_{v[1]} (\omega\otimes w)= (\nabla_{v[1]}\omega)\otimes w=( \iota_{\alpha(v)})\omega\otimes w=\iota_{\alpha(v)}(\omega\otimes w),
\end{equation}
where $\omega\in \Omegag $ and $w\in W$.
By Equation \eqref{Atiyah cocycle}, the $\delta_\alpha$-twisted Atiyah cocycle can be computed:
\begin{eqnarray}\nonumber
\At_\E^{\nabla }(v[1],w)&=&- d^W_\tot (\nabla _{v[1]} w) + \nabla _{d^W_\tot (v[1])}w + (-1)^{\abs{v[1]} }\nabla _{v[1]}(d^W_\tot  w) \\\nonumber
&=&(-1)^{\abs{v}-1}\nabla _{v[1]}(d^W_{\CE}  w)=
(-1)^{\abs{v}-1} \langle  \alpha (v)|d^W_{\CE}  w \rangle\\~\label{Eqt:AtEexplicit}
&=& (-1)^{\abs{v}-1}   \alpha (v) \triangleright w.
\end{eqnarray}

In conclusion, we establish the following proposition.
\begin{prop}~\label{Prop:Atiyahcocyclefromdelta}
With settings as above, the $\delta_\alpha$-twisted Atiyah class of $\E=\Omegag( W)$
\[
\At_\E^{\delta_\alpha}  \in H^1(\A, \Hom_\A(\B \otimes_\A \E,\E))=H^1_{\tot}( \g,\Hom_\k(V[1]\otimes W,W) )
\]
is the cohomology class of the particular $1$-cocycle $m\in C^1(\g,\Hom_\k(V[1]\otimes W,W))$ which, as a linear map (of degree $+1$)
\[
m \colon V[1]\otimes W \to \Omegag( W),
\]
is defined by
\[
 v[1]\otimes w \mapsto (-1)^{\abs{v}-1}   \alpha (v) \triangleright w,\qquad \forall v\in V, w\in W.
 \]
\end{prop}

Note that $\At_\E^{\delta_\alpha}$ depends only on the homotopy class of $\delta_\alpha$, and hence only on the dg \LP class $(V, [\alpha])$, according to Proposition \ref{Prop:homotopy-homotopy}.

Let $A$ be a $\g$-algebra, i.e., $A$ is a commutative  algebra equipped with a  $\g$-module structure such that $\g$-acts by derivations. Then $C=\Omegag( A )$ is endowed with a dg $\A=\Omegag$-algebra structure.
Furthermore, we have the following consequence of Proposition \ref{maincoro}.

\begin{prop}~\label{prop:LeibnizgradedLPmodule}	
 \begin{compactenum}
  \item The (degree $(-1)$ shifted) cohomology space $$H^\bullet(\A,C\otimes \B[-1])=H^\bullet_{\tot} ( \g,A\otimes V )$$ is a Leibniz algebra (over $H^\bullet(\A) =H^\bullet_{\CE}(\g)$), whose bracket ${\Bigl[-,-\Bigr]}_{\delta}$ is induced by the $\delta_\alpha$-twisted Atiyah class of $\B[-1]=\Omegag( V )$:
		\begin{eqnarray*}
		&&	\At_{\B[-1]}^{\delta_\alpha}   \in H^1(\A, \Hom_\A(\B \otimes_\A \B[-1],\B[-1]))\\ &=& H^0(\A, \Hom_\A(\B[-1] \otimes_\A \B[-1],\B[-1]))
		=H^0_{\tot}( \g,\Hom_\k(V \otimes V ,V ) ).
\end{eqnarray*}
		Moreover, $\At_{\B[-1]}^{\delta_\alpha}$ is the cohomology class of the particular $0$-cocycle
\[
r \in C^0( \g,\Hom_\k(V \otimes V ,V ) )
\]
 which, as a linear map $r \colon V \otimes V \to \Omegag(V)$ is defined by
\[
 v \otimes v' \mapsto    \alpha (v) \triangleright v',\qquad \forall v,v'\in V.
 \]		
\item There exists a representation of the Leibniz algebra $H^\bullet_{\tot} ( \g,A\otimes V ) $ on the cohomology space $H^\bullet(\A,C\otimes\E)=H^\bullet_{\tot} ( \g,A\otimes W )$, with the action map $-\triangleright-$ induced by the $\delta_\alpha$-twisted Atiyah class of $\E=\Omegag( W)$. In other words, $-\triangleright-$ is induced by the $0$-cocycle
\[
r\in C^0( \g,\Hom_\k(V \otimes W ,W ) )
\]
which, as a linear map  $c \colon V \otimes W \to \Omegag( W) $, is defined by
\[
 v \otimes w \mapsto    \alpha (v) \triangleright w,\qquad \forall v \in V, w\in W.
 \]
\end{compactenum}	
\end{prop}

\begin{Rem} The trivial case $A=\k$ does not yield anything interesting because the said Leibniz bracket in $H^\bullet_{\tot} ( \g,  V )$ and module structure on $H^\bullet_{\tot} ( \g,  W )$ are both zero. For a nontrivial case, following the idea of Kapranov \cite{Kap}, we take   $A=S(V^\vee)$, the symmetric algebra of $V$ (concentrated in degree $0$).
	Then $H^\bullet_{\tot} ( \g,S(V^\vee)\otimes V )$ admits a Leibniz bracket which is generally nontrivial. 	In fact, consider the $\partial_\A$-closed element $\id_{V}\in S(V^\vee)\otimes V$. One can verify that
	${\Bigl[\id_{V} ,\id_{V}\Bigr]}_{\delta}$ is identically $\At_{\B[-1]}^{\delta_\alpha}$.
\end{Rem}

\section{Kapranov sh Leibniz algebras arising from dg \LP modules}~\label{Sec:KapranovLeibnizetc}
Sh (short for strongly homotopy) Leibniz algebras, also known as Leibniz$_\infty$-algebras,   Leibniz$_\infty[1]$-algebras, or Loday infinity algebras, are homotopy replacements of Leibniz algebras or Loday algebras introduced by Loday~\cite{Loday}. The category of sh Leibniz algebras  introduced by Ammar and Poncin~\cite{AP} enjoys some nice phenomena as its subcategory of sh Lie algebras~\cite{LS}, like the minimal model theorem proved by Ammar and Poncin in~\emph{op.cit.}
In our earlier work \cite{CLX}, Kapranov's construction of an $L_\infty$-algebra structure~\cite{Kap} and Chen-Sti\'{e}non-Xu's construction of a Leibniz$_\infty[1]$-algebra structure~\cite{CSX} were generalized to the setting of a dg derivation $\A \xrightarrow{\delta} \Omega$ of a dg algebra $\A$. In this section, we continue to investigate such kinds of objects which stem from dg \LP modules. .

\subsection{Sh Leibniz algebras and related notions}
 In academic literature, the terms Sh Leibniz algebra and Leibniz$_\infty$-algebra are interchangeable. In this note, we opt to use the equivalent terminology of Leibniz$_\infty[1]$-algebras \cite{CSX}.

 \begin{Def}~\label{Def:Leibnizinfinity}
A Leibniz$_\infty[1]$-algebra over $\k$ is a graded vector space $N = \oplus_{n \in \Z}N^n$, together with a sequence $\{\lambda_k\colon \otimes^k N \rightarrow N\}_{k \geqslant  1}$ of degree $(+1)$ and multilinear maps satisfying
\begin{align*}
	&\sum_{i+j=n+1}\sum_{k=j}^{n}\sum_{\sigma \in \sh(k-j,j-1)}\epsilon(\sigma)(-1)^{\abs{v_{\sigma(1)}}+\cdots+ \abs{v_{\sigma(k-j)}}} \notag \\
	&\qquad\qquad\lambda_{i}\big(v_{\sigma(1)},\cdots,v_{\sigma(k-j)},\lambda_{j}(v_{\sigma(k-j+1)}, \cdots,v_{\sigma(k-1)}, v_k), v_{k+1},\cdots, v_n\big)=0,
\end{align*}
for all $n \geqslant  1$ and all homogeneous elements $v_i \in N$, where $\sh(p,q)$ denotes the set of $(p,q)$-shuffles ($p,q \geqslant  0$), and $\epsilon(\sigma)$ is the Koszul sign of  $\sigma$.
\end{Def}
\begin{Def}
A  morphism of Leibniz$_\infty[1]$-algebras from $(N, \{\lambda_k\}_{k\geqslant 1})$ to $(N^\prime, \{\lambda^\prime_k\}_{k\geqslant 1})$ is a sequence $\{f_k\colon N^{\otimes k} \rightarrow N^\prime\}_{k \geqslant  1}$ of degree $0$ and  multilinear maps satisfying the following compatibility condition:
\begin{align*}
 &\quad \sum_{k+p \leqslant  n-1}\sum_{\sigma \in \sh(k,p)}\epsilon(\sigma)(-1)^{\dagger^{\sigma}_k}f_{n-p}\big(b_{\sigma(1)},\cdots,b_{\sigma(k)}, \lambda_{p+1}(b_{\sigma(k+1)}, \cdots,b_{\sigma(k+p)},b_{k+p+1}),\cdots,b_{n}\big) \notag \\
&= \sum_{q \geqslant  1}\sum_{\substack{I^1\cup\cdots\cup I^q = \mathbb{N}^{(n)} \\ I_1,\cdots,I_q \neq \emptyset \\ i_{\abs{I^1}}^1 < \cdots < i^q_{\abs{I^q}}}}\epsilon(I^1,\cdots,I^q) \lambda_q'\big(f_{\abs{I^1}}(b_{I^1}), \cdots,f_{\abs{I^q}}(b_{I^q})\big),
\end{align*}
for all $n \geqslant  1$, where $\dagger^{\sigma}_k = \sum_{i=1}^k\abs{b_{\sigma(i)}}$, $I^j = \{i^j_1 < \cdots < i^j_{\abs{I^j}}\} \subset \mathbb{N}^{(n)} = \{1,\cdots,n\}$, and $(b_{I^j}) = (b_{i^j_1},\cdots,b_{i^j_{\abs{I^j}}})$ for all $1\leqslant  j \leqslant  q$.
\end{Def}
In the above definition, the first component $f_1\colon (N,\lambda_1) \rightarrow (N^\prime, \lambda_1^\prime)$, called the tangent morphism, is a morphism of cochain complexes. We call the Leibniz$_\infty[1]$ morphism $f_\bullet\colon (N, \lambda_\bullet) \rightarrow (N^\prime, \lambda^\prime_\bullet)$ an \textbf{isomorphism} if $f_1$ is an isomorphism.
In fact, there is a standard way to find its inverse
$f^{-1}_\bullet\colon (N^\prime, \lambda^\prime_\bullet) \rightarrow (N, \lambda_\bullet)$ (see \cite{AP}).

\begin{Def}~\label{Def:Leibnizinfinitymodule}
Let $(N,\{\lambda_k\}_{k\geqslant 1})$ be a Leibniz$_\infty[1]$-algebra.
An $(N,\{\lambda_k\}_{k\geqslant 1})$-module is a graded vector space $U$ together with a sequence $\{\mu_k\colon N^{\otimes(k-1)} \otimes U \rightarrow U\}_{k \geqslant  1}$ of degree $(+1)$ and multilinear maps, called actions, satisfying the identities
\begin{align*}%~\label{Eq:Leib module}
	&\sum_{i+j=n+1}\sum_{k=j}^{n}\sum_{\sigma \in \sh(k-j,j-1)}\epsilon(\sigma)(-1)^{\dagger^\sigma_{k-j}}  \notag\\
	&\qquad\mu_{i}\big(v_{\sigma(1)},\cdots,v_{\sigma(k-j)},\lambda_{j}(v_{\sigma(k-j+1)}, \cdots,v_{\sigma(k-1)}, v_k), v_{k+1},\cdots, v_{n-1},w \big) \notag \\
	&\qquad+\sum_{1\leqslant  j \leqslant  n}\sum_{\sigma \in \sh(k,j)}\epsilon(\sigma)(-1)^{\dagger^\sigma_{n-j}} \mu_{i}\big(v_{\sigma(1)},\cdots,v_{\sigma(n-j)},\mu_{j}(v_{\sigma(n-j+1)}, \cdots,v_{\sigma(n-1)}, w)\big)=0,
\end{align*}
for all $n \geqslant  1$ and all homogeneous vectors $v_1,\cdots,v_{n-1} \in N, w \in U$, where $\dagger^\sigma_j = \abs{v_{\sigma(1)}}+\cdots+ \abs{v_{\sigma(j)}}$ for all $j \geqslant  0$.
\end{Def}

In this note, we are interested in Leibniz$_\infty[1]$-algebras over a dg algebra $\A$, or Leibniz$_\infty[1]$ $\A$-algebras.
\begin{Def}~\label{Def:LeibnizinfinityA}
A Leibniz$_\infty[1]$ $\A$-algebra is a Leibniz$_\infty[1]$-algebra $(N, \{\lambda_k\}_{k\geqslant 1})$ (in the category of Leibniz$_\infty[1]$-algebras over $\k$) such that the cochain complex $(N,\lambda_1)$ is a dg $\A$-module and all higher brackets $\lambda_k\colon \otimes^k N \rightarrow N$ ($k\geqslant  2$) are  $\A$-multilinear.
	
A morphism of Leibniz$_\infty[1]$  $\A$-algebras from $(N, \{\lambda_k\}_{k\geqslant 1})$ to $(N', \{\lambda'_k\}_{k\geqslant 1})$ is given by a morphism $\{f_k \colon N^{\otimes k}\rightarrow N'\}_{k \geqslant  1}$ of Leibniz$_\infty[1]$ $\k$-algebras such that all structure maps $\{f_k\}_{k\geqslant  1}$ are  $\A$-multilinear. In particular, its tangent morphism $f_1\colon (N,\lambda_1) \rightarrow (N^\prime,\lambda^\prime_1)$ is a dg $\A$-module morphism.
	
Such a morphism $f_\bullet\colon (N,\lambda_\bullet) \rightarrow (N',\lambda'_\bullet)$ is called an \textbf{isomorphism} if its tangent morphism $f_1$ is an isomorphism of dg $\A$-modules.
\end{Def}
Denote the category of Leibniz$_\infty[1]$ $\A$-algebras by $\CATLeibnizoneA$. It is a subcategory of the category of Leibniz$_\infty[1]$-algebras over $\k$.
There are analogous notions of modules of a Leibniz$_\infty[1]$ $\A$-algebra and their morphisms.

\begin{Def}~\label{Def:LeibnizinfinityAmodule}
Let $(N,\{\lambda_k\}_{k\geqslant 1})$ be a Leibniz$_\infty[1]$ $\A$-algebra.
A $(N,\{\lambda_k\}_{k\geqslant 1})$ $\A$-module is a $(N,\{\lambda_k\}_{k\geqslant 1})$-module $(U,\{\mu_k\}_{k\geqslant 1})$ (in the category of Leibniz$_\infty[1]$-modules over $\k$)
such that $(U,\mu_1)$ is a dg $\A$-module and all higher structure maps $\{\mu_k\}_{k\geqslant  2}$ are  $\A$-multilinear.
\end{Def}

\subsection{The Kapranov functor} We inherit the previous notations in Section \ref{Sec:twistedAtiyahclassdgderivation}.
Let $\A \xrightarrow{\delta} \Omega$ be a dg derivation and $\B$ the dual dg $\A$-module of $\Omega$.

\subsubsection{Kapranov sh Leibniz algebras}
Choose a $\delta$-connection $\nabla$ on $\B$. We introduce a sequence of degree $(+1)$ maps
$ \mathcal{R}^\nabla_k\colon \B^{\otimes k} \rightarrow \B, k \geqslant  1 $
which are defined as follows:
\begin{itemize}
	\item The first one, also known as the unary  map, is $\mathcal{R}^\nabla_1 = \partial_\A\colon \B \rightarrow \B$;
	\item The second one, also known as the binary bracket, is $\mathcal{R}^\nabla_2:=\At_\B^\nabla$,   the associated twisted Atiyah cocycle;
	\item The higher operations $\{\mathcal{R}^\nabla_{k+1}\colon \B^{\otimes(k+1)} \rightarrow \B\}_{k\geqslant 2}$ are defined recursively by
	$\mathcal{R}^\nabla_{k+1}=\nabla(\mathcal{R}^\nabla_{k})$.
 \end{itemize}
Explicitly, we have
	\begin{align}\nonumber
		 \mathcal{R}^\nabla_{k+1}(b_0,b_1,\cdots,b_{k}) &= (-1)^{\abs{b_0}}[\nabla_{b_0},\mathcal{R}^\nabla_k] (b_1,\cdots,b_k)& \\\nonumber
		&= (-1)^{\abs{b_0}} \nabla_{b_0}\circ \mathcal{R}^\nabla_k(b_1,\cdots,b_k)\\~\label{Rnabla}
		&\qquad - \sum_{i=1}^k
		(-1)^{\abs{b_0}(\abs{b_1}+\cdots+\abs{b_{i-1}})}   \mathcal{R}^\nabla_k(b_1,\cdots,b_{i-1},\nabla_{b_0}b_i,b_{i+1},\cdots, b_k) , 		
	\end{align}
for all $ b_0,\cdots,b_k \in \B$.
\begin{prop}~\label{Thm: leibniz infty}
The $\A$-module $\B$, together with the sequence of operators $\{\mathcal{R}^\nabla_k\}_{k\geqslant  1}$, is a Leibniz$_\infty[1]$ $\A$-algebra.
\end{prop}

This method originates from Kapranov's construction of $L_\infty$-algebra structure on the shifted tangent complex $\Omega_X^{0,\bullet-1}(T^{1,0}X)$ of a compact K\"{a}hler manifold $X$~\cite{Kap}. For this reason, we call $(\B,\{\mathcal{R}^\nabla_k\}_{k\geqslant  1})$ the \textbf{Kapranov Leibniz}$_\infty[1]$ $\A$-algebra.
The Leibniz$_\infty[1]$ $\A$-algebra $(\B,\{\mathcal{R}^\nabla_k\}_{k\geqslant  1})$ will be denoted by $\Kap^c(\delta)$.
Here the superscript $c$ is to remind the reader that this Leibniz$_\infty[1]$ $\A$-algebra is defined via a particular $\delta$-connection on $\B$.
	
Moreover, the assignment of a {Kapranov Leibniz}$_\infty[1]$  $\A$-algebra to each pair of a dg derivation $\A \xrightarrow{\delta} \Omega$ and a $\delta$-connection $\nabla$ on $\B$ is functorial as described next.
Let $\phi$ be a morphism from $\A \xrightarrow{\delta^\prime} \Omega^\prime$ to $\A \xrightarrow{\delta} \Omega$ in the category $\CATderivationA$ of dg derivations (see Definition~\ref{Def:morphismDerA}).
Let $\B=\Omega^\vee$ and $\B'=(\Omega')^\vee$ be their dual dg $\A$-modules. For a $\delta$-connection $\nabla$ on $\B$ and a $\delta^\prime$-connection $\nabla^\prime$ on $\B^\prime$, there exists a morphism $f_\bullet = \{f_k\}_{k \geqslant  1}$ of Kapranov Leibniz$_\infty[1]$ $\A$-algebras from $(\B,\{\mathcal{R}^{\nabla}_k\}_{k\geqslant 1})$ to $(\B^\prime, \{\mathcal{R}^{\nabla^\prime}_k\}_{k\geqslant 1})$, whose first map is $f_1=\phi^\vee$.
In other words, we have the following commutative diagram
\[
	\begin{tikzcd}
		(\A \xrightarrow{\delta^\prime} \Omega^\prime) \ar{d}{\phi} \ar{r}{\Kap^{c}} & (\B^\prime,\{\mathcal{R}^{\nabla^\prime}_k\}_{k\geqslant 1}) \\
		(\A \xrightarrow{\delta} \Omega) \ar{r}{\Kap^c} & (\B,\{\mathcal{R}^{\nabla}_k\}_{k \geqslant  1}). \ar{u}{\Kap^c(\phi) = f_\bullet}
	\end{tikzcd}
\]

As a consequence, the Kapranov's construction defines a contravariant functor
\[
  \Kap^c\colon  \CATderivationA \to \CATLeibnizoneA
\]
from the category $\CATderivationA$ of dg derivations of $\A$ to the category $\CATLeibnizoneA$ of Leibniz$_\infty[1]$-algebras.
		
We need to stress the independence from  choice of connections in the definition of Kapranov functors.
For a dg derivation $\A \xrightarrow{\delta} \Omega$, suppose that we have another $\delta$-connection $\widetilde{\nabla}$ on $\B = \Omega^\vee$.
Denote the corresponding Kapranov Leibniz$_\infty[1]$ $\A$-algebra by $\Kap^{\tilde{c}}(\delta) = (\B,\mathcal{R}_k^{\widetilde{\nabla}})$.
It turns out that there exists a \textit{natural equivalence} between the two Kapranov functors $\Kap^c$ and $\Kap^{\tilde{c}}$ with respect to different connections.
By this natural equivalence, we are allowed to drop the superscript $c$.

A nice property is that the isomorphism class of Kapranov Leibniz$_\infty[1]$ $\A$-algebras arising from dg derivations depends only on their homotopy class.
\begin{prop}~\label{Thm: homotopy invariance}
Let $\delta \sim \delta^\prime$ be homotopic $\Omega$-valued dg derivations of $\A$. Then there exists an isomorphism $\{g_k\}_{k\geqslant 1}$ sending the Kapranov Leibniz$_\infty[1]$ $\A$-algebra $\Kap(\delta^\prime) = (\B,\{\mathcal{R}^{\nabla^\prime}_k\}_{k\geqslant 1})$ (with respect to a $\delta^\prime$-connection $\nabla^\prime$) to $\Kap(\delta) = (\B,\{\mathcal{R}^{\nabla}_k\}_{k\geqslant 1})$ (with respect to a $\delta$-connection $\nabla$).
\end{prop}

Let $\A \xrightarrow{\delta} \Omega$ be a dg derivation of $\A$ and $\E$ a dg $\A$-module. By similar arguments as above, $\E$ carries a Leibniz$_\infty[1]$ $\A$-module structure over $\Kap(\delta)$. Indeed, there exists a functor, also called the \textbf{Kapranov functor}, from the category $\DG\A$ of dg $\A$-modules to the category of Leibniz$_\infty[1]$ $\A$-modules over $\Kap(\delta)$.

\subsection{The Kapranov functor applied to dg \LP modules}
Let us fix a {dg \LP module} $(V, \alpha )$ throughout this section. Consider the dg algebra $\A=\Omegag $ and its module $\Omega$ $=\Omegag(  V^\vee[-1])$.
The $\A=\Omegag $-dual of $\Omega$ is clearly $\B=\Omegag(  V[1])$.

By Proposition \ref{prop1}, the {dg \LP module} $(V, \alpha )$ corresponds to a dg derivation $ \delta_\alpha$ of  $\A=\Omegag  $ valued in $\Omega=\Omegag(  V^\vee[-1])$ (the relationship is shown in Equation \eqref{deltaandalphak}).
We take the most naive $\delta_\alpha$-connection $\nabla$ on $\B=\Omega^\vee=\Omegag(  V[1])$ such that:
\begin{equation}~\label{Eqt:naivedeltaconnection2}
	\nabla_{v[1]} (\omega\otimes v'[1])= (\nabla_{v[1]}\omega)\otimes v'[1]=( \iota_{\alpha(v)}\omega)\otimes v'[1]=\iota_{\alpha(v)}(\omega\otimes v'[1]),
\end{equation}
where $\omega\in \Omegag $, $v,v'\in V$. Recall that we have used such types of naive $\delta_\alpha$-connections in Equation \eqref{Eqt:naivedeltaconnection}.

\begin{Thm}~\label{Main Thm}
Let $(V, \alpha )$ be a dg \LP module over $\g$, $\delta_\alpha\colon   \Omegag   \rightarrow \Omegag( V^\vee[-1])$ the associated dg derivation, and $\nabla$ as in Equation \eqref{Eqt:naivedeltaconnection2} the $\delta_\alpha$-connection on $\B=\Omegag(  V[1])$.
Then by the Kapranov functor, we obtain a Leibniz$_\infty[1]$ $\Omegag  $-algebra
$\Kap(\delta_\alpha) = (\B,\{\mathcal{R}^{\nabla}_k\}_{k\geqslant 1})$ on the $\Omegag $-module $\B $ whose structure maps $\{\mathcal{R}^{\nabla}_k\}_{k\geqslant 1}$ are described as follows:
\begin{compactenum}
  \item The unary bracket $\mathcal{R}^{\nabla}_1$ is the   total differential $d^{V[1]}_\tot =\dCE ^{V[1]} +d^{V[1]} $ of the dg $\g$-module $V[1]$;
  \item The binary bracket $\mathcal{R}^{\nabla}_2$ is the   $\delta_\alpha$-twisted Atiyah cocycle of $\B$ described by Proposition \ref{Prop:Atiyahcocyclefromdelta}.
  In other words, it is  $\Omegag  $-bilinear and extended from the degree $(+1)$ map
	\begin{align*}
		\mathcal{R}^{\nabla}_2\colon  V[1] {\otimes  }V[1] &\rightarrow \Omegag(  V[1]),\\
		v[1]\otimes v'[1] & \mapsto (-1)^{\abs{v}-1}   \alpha (v) \triangleright v'[1],\qquad \forall v,v'\in V.
	\end{align*}
  \item For all $n\geqslant  2$, the $( n+1)$-bracket $\mathcal{R}^{\nabla}_{n+1}$ is $\Omegag  $-multilinear and extended from the degree $(+1)$ map
     \begin{align*}
		\mathcal{R}^{\nabla}_{n+1}\colon (V[1])^{\otimes (n+1)} &\rightarrow \Omegag(  V[1])
	\end{align*}
	defined by
\begin{eqnarray}\nonumber
	\mathcal{R}^{\nabla}_{n+1}(v_0[1],\cdots, v_{n }[1])&=&(-1)^{\abs{v_0}+\cdots+\abs{v_{n-1}}-n}\bigl( \iota_{\alpha(v_0)}\circ  \iota_{\alpha(v_1)} \circ \cdots  \circ \iota_{\alpha(v_{n-2})}\alpha(v_{n-1})\bigr)\triangleright v_n[1],\\\label{Eqt:Rnexplicit1}
	&&
\end{eqnarray}
for all $  v_i\in V $.
\end{compactenum}
\end{Thm}
\begin{proof}
We only need to show the third statement. By its construction Equation \eqref{Rnabla}, $\mathcal{R}^{\nabla}_{n+1}$ is determined inductively by the following formula:
\[
	\mathcal{R}^{\nabla}_{n+1}(v_0[1],\cdots, v_{n }[1])=(-1)^{\abs{v_0}-1}\iota_{\alpha(v_0)} \mathcal{R}^{\nabla}_{n }(v_1[1],\cdots, v_{n }[1]),\qquad \forall v_i\in V.
\]
Here we used the fact that $\nabla_{v_0[1]}v_i[1]=0$. Then an easy induction yields the desired Equation~\eqref{Eqt:Rnexplicit1}.
\end{proof}
\begin{Rem}
In the particular case that $V$ is concentrated in degree $0$, say, $V=V^0=G $ and hence $d^V=0$, then $\alpha$ consists of only one map $\alpha_0=X\colon  G  \to  \g $
and $V[1]$ is concentrated in degree $(-1)$. The dg \LP module $(V,\alpha)$ is indeed an ordinary \LP module $\bigl(G\xrightarrow{X} \g \bigr)$.
By Equation \eqref{Eqt:Rnexplicit1}, all higher brackets $\mathcal{R}^{\nabla}_{n}$ vanish for $n\geqslant  3$. So the Kapranov Leibniz$_\infty[1]$ $\Omegag $-algebra structure on $\B=\Omegag( V[1] )=\Omegag(G[1])$ consists of two maps $ \dCE ^{G[1]} $ and $\mathcal{R}^{\nabla}_{2}$.
In other words, $\Omegag( G[1] )$ is a dg Leibniz$[1]$ algebra.
The pure degree $(-1)$ part of which, i.e., $G[1]$, is a Leibniz$[1]$ algebra with the bracket (by (2) of Theorem \ref{Main Thm}):
\begin{equation}\label{Eqt:temptemp}
\mathcal{R}^{\nabla}_2(g[1],g'[1])=-X(g)\triangleright g'[1],\quad \forall g,g'\in V^0=G,
\end{equation}
and hence $G$ is a Leibniz algebra. This recovers the original construction of Leibniz algebra structure on $V$ by Loday and Pirashvili \cite{LP}.
\end{Rem}

As an application, we consider an ordinary \LP module $\bigl(G\xrightarrow{X} \g \bigr)$.	
Let $V=(V^\bullet, d^V)$ be a non-negative dg $\g$-module which is also a resolution of $G$. Hence we can identify  $G$ directly with the sub $\g$-module $\ker(d_0^V)$ of $V^0$.
By Theorem \ref{Thm:foundamentalone} for the existence of dg lift of a \LP module, to the {\LP module } $(G,X)$, there exists a dg lift $(V, \alpha)$ of $(G,X)$, which is unique up to homotopy.
%and the corresponding \LP class $(V, [\alpha] )$ is unique.
In particular we have a map $\alpha_0 \colon V^0\rightarrow \g$ which extends $X: G\to \g$.

Now, applying Theorem \ref{Main Thm} to $(V, \alpha )$, namely, the Kapranov functor to the associated dg derivation $\delta_\alpha\colon   \Omegag  \rightarrow \Omegag( V^\vee[-1])$,
we obtain a Leibniz$_\infty[1]$ $\Omegag  $-algebra
$\Kap(\delta_\alpha) = (\B,\{\mathcal{R}^{\nabla}_k\}_{k\geqslant 1})$ on the $\Omegag $-module $\B= \Omegag(V[1])$.
The resulting Leibniz$_\infty[1]$ $\Omegag$-algebra structure is unique up to isomorphism (independent of the choice of $\alpha$), hence it depends only on the given \LP module structure $X:~G\rightarrow  \g  $. We observe that  the binary bracket $\mathcal{R}^{\nabla}_2$, when restricted to $G[1]=\ker(d_0^V)[1]\subset V^0[1]$, is again given by Equation \eqref{Eqt:temptemp}, which matches exactly with the Leibniz algebra structure on $G$ by Loday and  Pirashvili \cite{LP}. Thereby, the   Leibniz$_\infty[1]$ $\Omegag  $-algebra
$  (\Omegag(V[1]),\{\mathcal{R}^{\nabla}_k\}_{k\geqslant 1})$ can be regarded as a  resolution of the dg Leibniz$[1]$ algebra  $\Omegag( G[1] ) $.

We summarize  the  above discussions into the following proposition which is a  supplement to the previous Theorem \ref{Thm:foundamentalone} .

\begin{prop} 	
	Let $G$ be a $\g$-module and $V=(V^\bullet, d^V)$ a resolution of $G$. To every {\LP module} $(G,X)$ over $\g$,
	there exists a Leibniz$_\infty[1]$ $\Omegag  $-algebra
	$  (\Omegag(V[1]),\{\mathcal{R}^{\nabla}_k\}_{k\geqslant 1})$ which extends the dg Leibniz$[1]$ algebra  $\Omegag( G[1] ) $ arising from the  {\LP module} $(G,X)$.
\end{prop}
We also have the following corollary which is a consequence of the functionality of the Kapranov functor and Proposition \ref{prop1}.

\begin{Cor}
Let $ f $ be a morphism of dg \LP modules from $(V, \alpha )$ to $(V', \alpha' )$ as in Definition \ref{Def:morphismweakLPmodules}.
Let $(\B=\Omegag( V[1]),\{\mathcal{R}^{\nabla}_k\}_{k\geqslant 1})$ and $(\B'=\Omegag( V'[1]),\{\mathcal{R}^{\nabla'}_k\}_{k\geqslant 1})$ be, respectively,
the corresponding Kapranov Leibniz$_\infty[1]$ $\Omegag$-algebra arising from $(V, \alpha )$ and $(V', \alpha' )$ as described by Theorem \ref{Main Thm}.
Then there exists a morphism of Kapranov Leibniz$_\infty[1]$ $\Omegag $-algebras from $(\B,\{\mathcal{R}^{\nabla}_k\}_{k\geqslant 1})$ to $(\B^\prime, \{\mathcal{R}^{\nabla^\prime}_k\}_{k\geqslant 1})$  whose first map $\B\to \B'$ stems from $f: \Omegag( V)\to \Omegag( V')$.
\end{Cor}

Finally, we state the following fact whose proof is analogous to that of Theorem \ref{Main Thm}.

\begin{Cor}
Maintain the same assumptions as in Theorem \ref{Main Thm}. Given any dg $\g$-module $W = (W^\bullet,d^W)$, the dg $\Omegag $-module $\E=(\Omegag( W ),d^W_\tot)$ admits a Leibniz$_\infty[1]$ $\Omegag $-module structure $  (\E,\{\mathcal{S} _k\}_{k\geqslant 1})$ over the Kapranov Leibniz$_\infty[1]$ $\Omegag  $-algebra $\Kap(\delta_\alpha) = (\B,\{\mathcal{R}^{\nabla}_k\}_{k\geqslant 1})$ described by Theorem \ref{Main Thm}.
The structure maps $\{\mathcal{S} _k\}_{k\geqslant 1}$ are given as follows:
\begin{compactenum}
	\item The unary action $\mathcal{S} _1$ is the   total differential $d^{W}_\tot =\dCE ^{W} +d^{W} $ of the dg $\g$-module $W$.
	\item The binary action $\mathcal{S} _2$ is the   $\delta_\alpha$-twisted Atiyah cocycle of $\E$ described by Proposition \ref{Prop:Atiyahcocyclefromdelta}. In other words, it is  $\Omegag  $-bilinear and extended from the degree $(+1)$ map
		\begin{align*}
			\mathcal{S} _2\colon V[1] {\otimes  }W &\rightarrow \Omegag(  W),\\
			v[1]\otimes w & \mapsto (-1)^{\abs{v}-1}   \alpha (v) \triangleright w,\qquad \forall v \in V, w\in W.
		\end{align*}
	\item For all $n\geqslant  2$, the $( n+1)$-action $\mathcal{S} _{n+1}$ is $\Omegag  $-multilinear and extended from the degree $(+1)$ map
         \begin{align*}
			\mathcal{S} _{n+1}\colon (V[1])^{\otimes n }\otimes W &\rightarrow \Omegag(  W )
		\end{align*}
		defined by
          \[
			\mathcal{S} _{n+1}(v_0[1],\cdots, v_{n-1 }[1],w)=(-1)^{\abs{v_0}+\cdots+\abs{v_{n-1}}-n}\bigl( \iota_{\alpha(v_0)}\circ  \iota_{\alpha(v_1)} \circ \cdots
              \circ \iota_{\alpha(v_{n-2})}\alpha(v_{n-1})\bigr)\triangleright w,
		\]
    for all $v_i\in V $, $w\in W$.
	\end{compactenum}
\end{Cor}

\section{Application to Lie algebra pairs}~\label{Sec:appLiepair}
In this section, we extend our discussion to the construction of a specific type of dg  \LP  modules derived from Lie algebra pairs (i.e. inclusions of Lie algebras). We aim to apply the comprehensive theory and methodology detailed in previous sections to accomplish two objectives. Firstly, we focus on calculating in detail the twisted Atiyah class associated with the dg  \LP modules. Additionally, we investigate the Leibniz$_\infty[1]$ algebra structures that emerge from this construction.

 \subsection{Dg \LP modules arising from Lie algebra pairs}~\label{Ex: Lie pair2}
 Given a Lie algebra pair $(\T,\g)$, i.e., two Lie algebras $\T$ and $\g$ such that $\g \subset \T$ is a Lie subalgebra, there is a short exact sequence of vector spaces
 \begin{equation}\label{SES of VS in Lie pair}
 0 \rightarrow \g \xrightarrow{i} \T \xrightarrow{\pr_{\T/\g}} \T/\g \rightarrow 0.
 \end{equation}
 Both vector spaces $\T$ and $\T/\g$ admit $\g$-module structures defined respectively by
 \begin{align*}
 	& x \triangleright m := [x,m]_\T, ~\\ \mbox{ and }~& x \triangleright \pr_{\T/\g}(m)  := \pr_{\T/\g}([x,m]_\T),
 \end{align*}
 for all $x \in \g, m \in \T$. The said $\g$-action on $\T$ is the adjoint action, and the action on $\T/\g$ is canonical.

We form a 2-term cochain complex $V = ( V^{0}\xrightarrow{d^V} V^1)$, where
\[
V^{0}=\T  , \quad V^1=\T/\g,
\]
and the differential $d^V$ is the projection map $\pr_{\T/\g}$. It is easy to see that $V$ is a dg representation of $\g$.

Fix a splitting of Sequence~\eqref{SES of VS in Lie pair}, namely a pair of linear maps $j\colon \T/\g \rightarrow \T$ and $\pr_\g\colon \T \rightarrow \g$ that satisfy
 \begin{align*}
 	\pr_\g \circ i  = \id_\g, &\qquad \pr_{\T/\g} \circ j = \id_{\T/\g},  ~\mbox{ and }~ i \circ \pr_\g + j \circ \pr_{\T/\g}  = \id_\T.
 \end{align*}
So we can decompose $\T  \cong \g  \oplus  \T/\g  $ as a vector space, and an element in $\T$ is written as $m=(m_x, m_b)$, where $m_x\in \g$, $m_b\in \T/\g$.
Via this decomposition, $\T/\g$ is considered as a sub vector space in $\T$ (but not as a sub $\g$-module).

We define two linear maps
\[
\alpha_0 :=\pr_\g \colon \T= V^{0} \rightarrow \g,\quad \mbox{i.e.,} \quad \alpha_0(m_x,m_b) = m_x,
\]
and
\[
\alpha_1\colon V^1 \otimes  \g  \rightarrow \g,\qquad \mbox{such that } \alpha_1(b \mid x )  = \pr_\g([ j(b),  i(x) ]_\T), ~\forall~ b\in \T/\g, x\in\g.
\]
Let us verify that these two maps give rise to a {dg \LP module} structure on $V$ over $\g$.
It suffices to examine Equation \eqref{equation of f} for $k=0$ and $k=1$.
For the $k=0$ case, we take elements $m = (m_x, m_b)\in\T=V^{0}, x  \in \g$, and compute
 \begin{align*}
 	&\quad \alpha_{1}(d_{V}(m) \mid x) - \alpha_0(x \triangleright m)+[x,\alpha_0(m)]_\g \\
 	&=\alpha_{1}(m_b\mid x )-\alpha_0([x ,m]_\T )+[x ,m_x]_\g \\
 	&=\pr_\g([m_b,x]_\T)-([x ,m_x]_\g + \pr_\g([x ,m_b]_\T))  + [x ,m_x]_\g = 0.
 \end{align*}
The $k=1$ case is similar because it amounts to check the following identity: for all $x_1,x_2\in \g$ and $b\in \T/\g=V^1$,
 \begin{align*}
 	\alpha_{1}\bigl(x_{2} \triangleright b \mid x_{1}  \bigr)-\alpha_{1}\bigl(x_{1} \triangleright b| x_{2} \bigr)-\alpha_1(b \mid [x_1,x_2]_\g)
 	&= \left[x_{2}, \alpha_{1}\bigl(b | x_{1} \bigr)\right]-[x_{1}, \alpha_{1}\bigl(b \mid x_{2} \bigr)],
 \end{align*}
for which we omit the details of verification.

In summary, we obtain a {dg \LP module} $(V,\alpha) = ( V^{0}\xrightarrow{d^V} V^1, \alpha)$ with $\alpha=\{\alpha_0,\alpha_1\}$ as above when a splitting $(j,\pr_\g)$ of Sequence~\eqref{SES of VS in Lie pair} is given.

Suppose that $(j^\prime,\pr^\prime_\g)$ is another splitting of Sequence~\eqref{SES of VS in Lie pair}.
Note that the difference between the two splittings defines a linear map $h_1= j-  j^\prime :~ V^1=\T/\g \to \g$. And accordingly, we have for all $m\in \T$,
\begin{eqnarray*}
 	\pr^\prime_\g(m)&=&m-j^\prime\circ \pr_{\T/\g}(m)=m-j \circ \pr_{\T/\g}(m)+h_1\circ \pr_{\T/\g}(m)\\
 	&=& \pr_\g(m)+h_1\circ \pr_{\T/\g}(m).
\end{eqnarray*}
Thus the {dg \LP module} structure maps corresponding to $(j^\prime, \pr^\prime_\g)$ are
\[
\alpha_0^\prime = \pr^\prime_\g=\pr_\g +h_1\circ \pr_{\T/\g}=\alpha_0+h_1\circ d^V,
\]
and
\begin{eqnarray*}
 	\alpha_1^\prime(b\mid x) &=& \pr_\g^\prime([j^\prime(b),i(x)]_\T)=\pr_\g ([j (b)-h_1(b),i(x)]_\T)+h_1\circ \pr_{\T/\g} ([j (b) ,i(x)]_\T)\\
 	&=& \alpha_1 (b\mid x)+[x,h_1(b)] - h_1(x\triangleright b).
\end{eqnarray*}
These relations evidently show that $h=\{h_1\}$ defines a homotopy between $\alpha=\{\alpha_0,\alpha_1\}$ and $\alpha^\prime=\{\alpha^\prime_0,\alpha^\prime_1\}$ by Proposition \ref{Prop:homotopyinh}.

We summarize the above results in the following proposition.
\begin{prop}
The above construction defines a \textit{canonical} {dg \LP module} class $(V, [\alpha]=[\alpha_0,\alpha_1])$ arising from a given Lie algebra pair $(\T,\g)$.
Moreover, it is sent by the functor $\H^0$ to the obvious \LP module
\[
\H^0(V) \xrightarrow{\H^0(\alpha)} \g,
\]
where $\H^0(V)=\ker (V^{0}\xrightarrow{d_0^V} V^1)=\g$ and $\H^0(\alpha)$ is the identity map on $\g$.
\end{prop}

We now provide a concrete example of Lie algebra pairs.
\begin{Ex}~\label{Ex:sl2}
 Consider $\T =\mathfrak{sl}_2(\mathbb{C})$ with the standard basis
 	\begin{align*}
 	h &= \Bigl(
 	\begin{array}{cc}
 	1 & 0 \\
 	0 & -1 \\
 	\end{array}
 	\Bigr), &    e &= \Bigl(
 	\begin{array}{cc}
 	0 & 1 \\
 	0 & 0 \\
 	\end{array}
 	\Bigr), &  f &= \Bigl(
 	\begin{array}{cc}
 	0 & 0 \\
 	1 & 0 \\
 	\end{array}
 	\Bigr)
 	\end{align*}
which is subject to the relations
 	\begin{align*}
 	[e,f] &= h, & [h,e] &= 2e, & [h,f] &= -2f.
 	\end{align*}
Let $\g $ be the Lie subalgebra $\g  = \mathrm{span}_{\C}\{h,e\} \subset \T =\mathfrak{sl}_2(\mathbb{C})$.
Then $V^1 = \T /\g $ is a one-dimensional vector space generated by $b$, the image of $f$ by projection. The $\g $-module structure on $V^1$ is given by
 	\begin{align*}
 	e \triangleright b  &= 0,  \quad\quad\quad \quad\quad\quad\mbox{and} &      h \triangleright b &= -2b.
 	\end{align*}
 The associated dg \LP module structure map $\alpha$ on the dg representation $V = \T  \oplus \T /\g $ of $\g $ is given by two linear maps ${\alpha_0}\colon \T  \rightarrow \g  $ and ${\alpha_1} \colon (\T /\g)  \otimes \g   \rightarrow \g  $ such that
 	\begin{align*}
 	{\alpha_0}(e ) = e ,  & \qquad\qquad {\alpha_0}(h ) = h ,  \qquad\qquad {\alpha_0}(f ) = 0,  \\
 	{\alpha_1}(b|e)  := \pr_\g ([f,e]) = -h,  & \qquad\mbox{ and }   \qquad  {\alpha_1}(b|h)   := \pr_\g ([f,h]) = 0.
 	\end{align*}
 \end{Ex}
Of course, one may come up with similar examples from $\mathfrak{sl}_n(\mathbb{C})$ and general semi-simple Lie algebras.

\subsection{Twisted Atiyah cocycles and Kapranov sh Leibniz algebras of Lie algebra pairs}
Let $(\T ,\g )$ be a Lie algebra pair.
We fix a splitting so that we treat $\T  = \g  \oplus \T /\g $ directly.
Consider the dg \LP module $(V,\alpha)$ discussed in the previous Section \ref{Ex: Lie pair2}.
For convenience, we introduce the following three maps:
\begin{align*}
\nabla \colon &\g \times \T/\g \to \T/\g, & \nabla_x b &:=\pr_{\T/\g} [x,b]_{\T},\\
\Delta \colon & \T/\g \times \g \to \g, & \Delta_b x &:=\pr_{ \g} [b,x ]_{\T},  \\
\triangleright \colon & \g \times \T \to \T, &  x\triangleright m &:= [x,m]_{\T},
\end{align*}
for all $m \in \T, b \in \T/\g $, and $x\in \g$.
The twisted Atiyah cocycle  $\mathcal{R}^{\nabla}_2\colon  V[1] {\otimes  }V[1]  \rightarrow \Omegag(V[1]) $ has the following terms:
\begin{compactenum}
 \item $V^{0}[1]\otimes V^{0}[1] \rightarrow  V^{0}[1]$-part   given by
	\[
(m_1[1],m_2[1])\mapsto \mathcal{R}^{\nabla}_2(m_1[1],m_2[1])=-[ \pr_{ \g }m_1,m_2]_{\T}[1]=-( \pr_{ \g }m_1)\triangleright m_2[1];
\]
 \item the $V^{0}[1]\otimes V^{1}[1] \rightarrow  V^{1}[1]$~-part   given by
	\[
(m_1[1],b_2[1])\mapsto \mathcal{R}^{\nabla}_2(m_1[1],b_2[1])=-\pr_{\T/\g}[ \pr_{ \g }m_1,b_2]_{\T}[1]=-\nabla_{\pr_{ \g }m_1}(b_2)[1];
\]
 \item the $V^{1}[1]\otimes V^{0}[1] \rightarrow  \g^\vee\otimes V^{0}[1]$-part   given by
	\[
(b_1[1],m_2[1])\mapsto \mathcal{R}^{\nabla}_2(b_1[1],m_2[1]),
\]
such that
  \[
  \iota_x \mathcal{R}^{\nabla}_2(b_1[1],m_2[1])= [ \pr_{ \g }[b_1,x]_\T ,m_2]_{\T}[1]=
      ( \Delta_{b_1}x)\triangleright m_2[1];
  \]
  \item  the $V^{1}[1]\otimes V^{1}[1] \rightarrow  \g^\vee\otimes V^{1}[1]$-part   given by
	\[
(b_1[1],b_2[1])\mapsto \mathcal{R}^{\nabla}_2(b_1[1],b_2[1]),
\]
such that
\[
 \iota_x \mathcal{R}^{\nabla}_2(b_1[1],b_2[1])= \pr_{\T/\g}[ \pr_{\g} [b_1,x]_\T ,b_2]_{\T}[1]=\nabla_{( \Delta_{ b_1} x)}b_2[1],
 \]
for all $m_1, m_2 \in \T=V^0, b_1, b_2\in \T/\g=V^1$, and all $x\in \g$.
\end{compactenum}
	
\begin{Rem}\label{Rem:originalandtwisted}
Compared with the original Atiyah cocycle of a Lie pair introduced by Calaque, C\u{a}ld\u{a}raru, and Tu \cite{CCT} (see also Chen-Sti\'{e}non-Xu's \cite{CSX}),
we see that the twisted Atiyah cocycle we obtain, i.e., $\mathcal{R}^{\nabla}_2$, has three more terms (1) $\sim$ (3).
Only its fourth component coincides with the original one.
In particular, if our Lie pair arises from a Lie bialgebra as in the setting of \cites{Hong},  the fourth component that we obtain recovers the computation done in~\emph{op.cit.}.
\end{Rem}	
We then consider the Kapranov Leibniz$_\infty[1]$-algebra given by Theorem~\ref{Main Thm} arising from the dg \LP module $(V,\alpha)$.
The unary bracket is standard. The binary bracket is already given by the above $\mathcal{R}^\nabla_2$.
By Equation \eqref{Eqt:Rnexplicit1}, we can also work out higher brackets $\mathcal{R}^{\nabla}_{n}$ for $n\geqslant3$.
 {For example, since $\mathcal{R}^{\nabla}_3$ is $\Omegag$-multilinear, it can be determined by extending   the following  maps:}
\begin{compactenum}
 \item the $V^{0}[1]\otimes V^{1}[1]\otimes V^{0}[1] \rightarrow    V^{0}[1]$-part   given by
  	\[
	(m_0[1],b_1[1],m_2[1])\mapsto \mathcal{R}^{\nabla}_3(m_0[1],b_1[1],m_2[1]) =   ( \Delta_{b_1}(\pr_\g m_0))\triangleright m_2[1];
	\]
 \item  the $V^{0}[1]\otimes V^{1}[1]\otimes V^{1}[1] \rightarrow    V^{1}[1]$-part   given by
  	\[
  (m_0[1],b_1[1],b_2[1])\mapsto  \mathcal{R}^{\nabla}_3(m_0[1],b_1[1],b_2[1])=  \nabla_{ \Delta_{ b_1} (\pr_\g m_0) }b_2[1];
  \]
 \item the $V^{1}[1]\otimes V^{1}[1]\otimes V^{0}[1] \rightarrow  \g^\vee\otimes V^{0}[1]$-part   given by
  	\[
  (b_0[1],b_1[1],m_2[1])\mapsto \mathcal{R}^{\nabla}_3(b_0[1],b_1[1],m_2[1])
  \]
  such that
   \[
   \iota_x \mathcal{R}^{\nabla}_3(b_0[1],b_1[1],m_2[1])=-( \Delta_{b_1}\Delta_{b_0}x)\triangleright m_2[1];
   \]
 \item  the $V^{1}[1]\otimes V^{1}[1]\otimes V^{1}[1] \rightarrow  \g^\vee\otimes V^{1}[1]$-part   given by
  \[
  (b_0[1],b_1[1],b_2[1])\mapsto \mathcal{R}^{\nabla}_3(b_0[1],b_1[1],b_2[1]),
  \]
 such that
  \[
  \iota_x \mathcal{R}^{\nabla}_3(b_0[1],b_1[1],b_2[1])=- \nabla_{( \Delta_{ b_1}\Delta_{ b_0} x)}b_2[1].
  \]
 \end{compactenum}
The interested reader will find that expressions of higher-order brackets $\mathcal{R}^\nabla_n$ (for $n\geqslant  4$) are completely similar to those of the    $n=3$ case.

We finally give the Kapranov Leibniz$_\infty[1]$-algebra arising from the dg \LP module $(V,\alpha)$ of Example \ref{Ex:sl2}.
\begin{Ex}
We use the same notations as introduced in Example \ref{Ex:sl2}. The dg representation $V = \T  \oplus \T /\g $ of $\g$ has basis $\{h,e,f\}$ for $\T$ (concentrated in degree $0$) and $\{b\}$ for $\T/\g$ (concentrated in degree $1$). Denote by $\{h^\vee,e^\vee\}$ the basis of $\g^\vee$ in duality to $\{h,e\}$ of $\g$.
With these vectors, the graded space $C^\bullet(\g ,V)$ is given by the exterior algebra generated by $\{h^\vee, e^\vee\}$ tensoring with $\{h,e,f,b\}$. 	
Let us write the Leibniz$_\infty[1]$-algebra structure maps on $C^\bullet(\g ,V)$ explicitly by giving their formulas on generators.
\begin{enumerate}
\item The unary map is the differential $d_\tot^V$ sending $h^\vee$ to $0$, $e^\vee$ to $(-2 h^\vee\wedge e^\vee)$, $h$ to $(-2 e^\vee\otimes e)$, $e$ to $2 h^\vee \otimes e$, $f$ to $(-2 h^\vee\otimes f+e^\vee\otimes h+b)$, and $b$ to $(-2h^\vee\otimes b)$.
\item For the binary bracket, we have
	\begin{align*}
				\mathcal{R}^{\nabla}_2(h,e)=-2e,\quad \mathcal{R}^{\nabla}_2(h,f)= 2f,\quad
				\mathcal{R}^{\nabla}_2(e,f)=-h,\quad
				\mathcal{R}^{\nabla}_2(f,e)=0,\quad\\
				\mathcal{R}^{\nabla}_2(h,b)=2b,\quad \mathcal{R}^{\nabla}_2(e,b)=0,\quad
				\mathcal{R}^{\nabla}_2(b,h)=0,\quad
				\mathcal{R}^{\nabla}_2(b,e)=-2e^\vee\otimes e,\quad\\
				\mathcal{R}^{\nabla}_2(b,f)=2e^\vee\otimes f,\quad\mathcal{R}^{\nabla}_2(b,b)=2e^\vee\otimes b.\qquad\qquad\qquad\qquad
	\end{align*}
\item The only nontrivial trinary bracket is given by:
			\[
\mathcal{R}^{\nabla}_3(e,b,e)=-2e,\mbox{ ~and~  }~\mathcal{R}^{\nabla}_3(e,b,f)=2f.
\]
			All other generating relations of $\mathcal{R}^{\nabla}_3(\cdot,\cdot,\cdot)$ are zero.
 \item All higher $\mathcal{R}^{\nabla}_{n}$  (for   $n\geqslant  4$) operations vanish.
		
\end{enumerate}
\end{Ex}

\appendix
\section{Proof of Theorem~\ref{Thm:basicdglift}}
Let us assume simply that $E=\H^0(V)$ and $F= \H^0(W)$.
We first prove the existence of the lift $f$, or $f_k \colon V^k\to \oplus_{l+s=k}\wedge^l \g^\vee\otimes W^s$, $k= 0, 1, \cdots,\mathrm{top}$.
Our objective is to construct these  $f_k$ so that they satisfy Equation \eqref{mor-dg}, or  	 	
\begin{equation}~\label{cons-lift}
f_{k+1}\circ d^V_k=-f_k\circ d_\CE^V+(d^W+d^W_\CE) \circ f_k, \quad k\geqslant  0. 		
\end{equation}
We will show the existence of $f_k$ by induction on $k$. The approach is as follows. If $f_k$ is given, then we can use Equation \eqref{cons-lift} to recursively define $f_{k+1}$ on
$\img  (d^V_k)\subset V^{k+1}$, provided that the right hand side of Equation \eqref{cons-lift} vanishes on $\ker (d^V_k)$. We will prove this fact and then define $f_{k+1}$ on $V^{k+1}$ by arbitrary extension (based on the assumption that every $V^{\bullet}$ is finite-dimensional).
	
Let $f_0 \colon V^0 \to W^0$ be an arbitrary extension of $\varphi $. To define $f_1 \colon V^1\to W^1\oplus \g^\vee \otimes W^0$, we must show that the right hand side of Equation \eqref{cons-lift} vanishes on $\ker (d^V_0)$, i.e.,
\[
 d^W_0\circ f_0^0(v)=0,
 \]
and
\[
 d^W_{\CE}\circ f_0^0(v)-f_0^0\circ \dCE ^V(v)=0,
 \]
for $v\in \ker (d_0^V)$. These identities are easy to check since $f_0^0$ restricts to the $\g$-module morphism $\varphi \colon \H^0(V)\to \H^0(W)$.
	
Now suppose that $f_0$, $f_1,\cdots, f_{n-1},f_{n}$ are defined for some $n\geqslant 1$, so that Equation \eqref{mor-dg} holds for all $k=1,\cdots, n-1$. Let us show that the right hand side of Equation \eqref{cons-lift} vanishes on $\ker (d^V_n)$, and then one defines $f_{n+1}$ accordingly.
We need the assumption that $(V, d^V)$ is acyclic, and hence we have $ \ker (d^V_n)=\img  (d^V_{n-1})$ in particular.
So for $v\in \ker (d^V_n)$, one can write $v=d^V_{n-1}(v_0)$, for some $v_0\in V^{n-1}$. Applying to $v$ the right hand side of Equation \eqref{cons-lift} where $k=n$, we have
\begin{eqnarray*}
&&-f_{n} \circ d_\CE^V (v) +d^W_\tot \circ f_n (v)\\
&=& -f_n \circ d_\CE^V\circ d^V(v_0) +d^W_\tot \circ f_n\circ d^V_{n-1} (v_0)\\
&=& f_n \circ d^V\circ d_\CE^V  (v_0)\\
&&\quad\quad +d^W_\tot \circ (
d^W_\tot \circ f_{n-1}-f_{n-1}\circ d^V_{\CE})(v_0)\qquad \mbox{by Eq.\eqref{atc} and the induction assumption}; \\
&=& (f_n\circ d^V+f_{n-1}\circ d^V_{\CE}-d^W_\tot \circ f_{n-1})\circ d^V_{\CE} (v_0)\\
&=& 0,\qquad \mbox{by   the induction assumption}.
\end{eqnarray*}
The statement of existence is thus justified.
	
We then prove that, if $f$ and $f' \colon V\rightsquigarrow W$ are both dg lifts of $\varphi$, then they are homotopic. It suffices to prove that, if $f$ is a dg lift of the trivial morphism $\varphi=0\colon  \H^0(V)\to \H^0(W)$ (i.e., $f|_{\H^0(V)=\ker(d^V_0) }=0$), then $f$ is homotopic to zero, i.e., there exists a degree $(-1)$ and $\Omegag $-linear map $h: \Omegag( V)\to \Omegag( W)$ such that
\begin{equation}~\label{Eqt:fh}
f  =   d^W_{\tot} \circ h + h \circ d^V_{\tot} \colon \Omegag( V)\to \Omegag( W).
\end{equation}
Let us assume that $h$ is generated by a family of linear maps
\[
h_k\colon V^k\to \oplus_{l+s=k-1}\wedge^l \g^\vee\otimes W^s, \quad k=1,2,\cdots,
\]
with $h_0=0$, and each $h_k$ extended to a map $\Omegag( V^k)\to \Omegag( W^{s})$, by setting
\[
 h_k(\omega \otimes v)=(-1)^p\omega\otimes h_k(v), \quad \omega\in \wedge^p \g^\vee,v\in V^k.
\]	
Then Equation \eqref{Eqt:fh} becomes the relation
\begin{equation}~\label{homotopy-k}
h_{k+1}\circ d^V_k=f_k-d^W_\tot\circ h_k-h_k\circ d^V_{\CE} \colon V^k\to \oplus_{l+s=k}\wedge^l \g^\vee\otimes W^s,
\end{equation}	
for all $k \geq 0$.
Again we use induction to construct these $h_k$s. The initial $h_0=0$ is given. The condition for $h_1$ is
\[
h_1\circ d^V_0=f_0.
\]
The construction of $h_1$ that meets this condition is clear because $f_0$ vanishes on $\ker d^V_0$.
	
Suppose that $h_0,h_1,\cdots, h_{n}$ ($n\geqslant  1$) are found and they all satisfy Equation \eqref{homotopy-k} for $k=0,1,\cdots, n-1$.
We then use Equation \eqref{homotopy-k} to define $h_{n+1}$ on $\img d^V_n$. Of course, one needs to examine that the right hand side of Equation \eqref{homotopy-k}, for $k=n$, yields zero when it is applied to $\ker d^V_n$. This is verified below.
	
For all  $v\in \ker (d^V_n)$, again by $(V,d^V)$ being acyclic, one can write $v=d^V_{n-1}(v_0)$, for some $v_0\in V^{n-1}$. Thus, we can compute
\begin{eqnarray*}
&& f_n(v)-d^W_\tot\circ h_n(v)-h_n\circ d^V_{\CE}(v)\\
&=& f_n\circ d^V_{n-1}(v_0) -d^W_\tot\circ h_n \circ d^V_{n-1}(v_0)-h_n\circ d^V_{\CE}\circ d^V_{n-1}(v_0)\\
&=& f_n\circ d^V_{n-1}(v_0) -d^W_\tot\circ h_n \circ d^V_{n-1}(v_0)+h_n\circ   d^V_{n-1}\circ d^V_{\CE}(v_0)\\
&=& f_n\circ d^V_{n-1}(v_0)-d^W_\tot\circ\circ (f_{n-1} -d^W_\tot\circ h_{n-1} -h_{n-1}\circ d^V_{\CE} )(v_0)\\
&&\qquad+(f_{n-1} -d^W_\tot\circ h_{n-1} -h_{n-1}\circ d^V_{\CE} )\circ  d^V_{\CE}(v_0)\\
&=& f_n\circ d^V_{n-1}(v_0)-d^W_\tot\circ f_{n-1}(v_0)+ f_{n-1}\circ d^V_{\CE}(v_0)\\
&=& f\circ d^V_\tot(v_0)-d^W_\tot\circ f_{n-1}(v_0)=0.
\end{eqnarray*}
Here we have used the induction assumption, and the fact that $f \colon V\rightsquigarrow W$ is a morphism of dg $\g$-modules.
	
This proves that $h_{n+1}|_{\img d^V_{n}}$ is well defined, and then we define $h_{n+1}$ by an arbitrary extension to $V^{n+1}$. By induction, we complete the proof.

{\bf Statements and Declarations}

{\bf Conflict of interest:} We hereby declare that this work has no related financial or non-financial conflict of interests.

\smallskip

{\bf Research Funds:} The research was partially supported by National Natural Science Foundation of China grants 12071241 (Chen), 11971282 (Qiao), 11901221 (Xiang), and 11501179 (Zhang).

\end{document}